\documentclass[a4paper]{amsart} 
\usepackage{microtype}

\usepackage{amsmath}
\usepackage{amssymb}
\usepackage{verbatim}
\usepackage{amsthm}
\usepackage{mathrsfs}
\usepackage{graphicx}
\usepackage{mathtools}
\usepackage{caption}
\usepackage{capt-of}
\usepackage{subcaption}
\usepackage{url}
\usepackage{epstopdf}
\usepackage{pgf,tikz}
\usepackage{bbm}
\usepackage{extarrows}
\usetikzlibrary{arrows}

\usepackage{algorithm,algpseudocode}
\usepackage{amssymb,amsthm}    
\usepackage{graphicx}          
\usepackage{amsmath}         
\usepackage{dsfont}
\usepackage{listings}
\usepackage{bbm}
\usepackage{color}
\usepackage{ dsfont }
\usepackage{enumerate}
 \usepackage{relsize}
\usepackage{xfrac}
\usepackage{changes}
\definecolor{darkblue}{rgb}{0,0,0.7}

\definecolor{darkgreen}{rgb}{0.01,0.75,0.24}

\def \Ee[#1]{\mathcal{E}^{\text{{#1}}}}
\def\R{\mathbb{R}}

\def\pa[#1,#2]{\frac{\partial {#1}}{\partial {#2}} }

\def\idom[#1,#2,#3]{\int_{#1}\hspace{1pt} {#2} \hspace{1pt} \text{d}{#3}}
\def\res[#1,#2]{\left.{#1}\right|_{#2}}

\def\var[#1,#2]{\langle \delta \mathcal{E}^{\text{{#1}}}({#2}),v\rangle}
\def\vars[#1,#2,#3]{\langle \delta^2\mathcal{E}^{\text{{#1}}}({#2})v,{#3}\rangle}
\def\vard[#1,#2,#3,#4]{\langle \delta\mathcal{E}^{\text{{#1}}}({#2})-\delta\mathcal{E}^{\text{{#3}}}({#4}),v\rangle}

\def\N{\mathbb{N}}

\DeclareMathOperator{\spann}{span}

\newcommand{\cO}{\mathcal{O}}

\newcommand{\dd}{ \,\textrm{d}}

\renewcommand{\Delta}{\triangle}

\newcommand{\bbR}{\mathbb R}

\DeclareMathOperator*{\argmax}{arg\,max}

\newcommand{\be}{\begin{equation}}
\newcommand{\en}{\end{equation}}
\newcommand{\ben}{\begin{equation*}}
\newcommand{\enn}{\end{equation*}}
\newcommand{\bea}{\begin{aligned}}
\newcommand{\ena}{\end{aligned}}

\def\ba#1\ena{\begin{align}#1\end{align}}
\def\ban#1\enan{\begin{align*}#1\end{align*}}

\newtheorem{prop}{Proposition}
\newtheorem{thm}{Theorem}
\newtheorem{rmk}{Remark}

\mathtoolsset{showonlyrefs=true}

\begin{document}

\author[P.~A. Guth] {Philipp A.~Guth}
\address{Institute of Mathematics, University of Mannheim, 68131 Mannheim, Germany}
\email{pguth@mail.uni-mannheim.de}

\author[C. Schillings] {Claudia Schillings}
\address{Institute of Mathematics, University of Mannheim, 68131 Mannheim, Germany}
\email{c.schillings@uni-mannheim.de}

\author[S. Weissmann] {Simon Weissmann}
\address{Institute of Mathematics, University of Mannheim, 68131 Mannheim, Germany}
\email{sweissma@mail.uni-mannheim.de}

\title{Ensemble Kalman filter for neural network based one-shot inversion}

\begin{abstract}
We study the use of novel techniques arising in machine learning for inverse problems. Our approach replaces the complex forward model by a neural network, which is trained simultaneously in a one-shot sense when estimating the unknown parameters from data, i.e.~the neural network is trained only for the unknown parameter. By establishing a link to the Bayesian approach to inverse problems, an algorithmic framework is developed which ensures the feasibility of the parameter estimate w.r.~to the forward model. We propose an efficient, derivative-free optimization method based on variants of the ensemble Kalman inversion. Numerical experiments show that the ensemble Kalman filter for neural network based one-shot inversion is a promising direction combining optimization and machine learning techniques for inverse problems.
\end{abstract}

\maketitle

\section{Introduction} 
Inverse problems arise in almost all areas of application, e.g.~biological problems, engineering and environmental systems. The integration of data can substantially reduce the uncertainty in predictions based on the model and is therefore indispensable in many applications. Advances in machine learning provide exciting potential to complement and enhance simulators for complex phenomena in the inverse setting. We propose a novel approach to inverse problems by approximating the forward problem with neural networks, i.e.~the computationally intense forward problem will be replaced by a neural network. However, the neural network when used as a surrogate model needs to be trained beforehand in order to guarantee good approximations of the forward problem for arbitrary inputs of the unknown coefficient. To reduce the costs associated with building the surrogate in the entire parameter space, we suggest to solve the inverse problem and train the neural network simultaneously in a one-shot fashion. From a computational point of view, this approach has the potential to reduce the overall costs significantly, as the neural network serves as a surrogate of the forward model only in the unknown parameter estimated from the data and not in the entire parameter space. The ensemble Kalman inversion (EKI) as a derivative-free optimizer is applied to the resulting optimization problem.   
\subsection{Background and literature overview}
The goal of computation is to solve the following inverse problem:  Recover the unknown parameter $u\in\mathcal X$ in an abstract model
\begin{equation}\label{eq:abstract_modela}
M(u,p) = 0
\end{equation}
from a finite number of observation of the state $p\in\mathcal V$ given by
\begin{equation}\label{eq:abstract_modelb}
O(p) = y\in\R^{n_y}\,,
\end{equation}
which might be subject to noise. We denote by $\mathcal X$, $\mathcal V$ and $\mathcal W$ Banach spaces. The operator $M:\mathcal X\times \mathcal V\to \mathcal W$ describes the underlying model, typically a PDE or ODE in the applications of interest. The variable $p$ denotes the state of the model and is defined on a domain $D\subset \mathbb R^d, \ d\in \mathbb N$. The operator $O:\mathcal V\to \R^{n_y}$ is the observation operator, mapping the state variables $p$ to the observations. 
Classical methods to inverse problems are based on an optimization approach of the (regularized) data misfit, see e.g.~\cite{EHN96, Kaltenbacher2008IterativeRM}. One-shot (all-at-once) approaches are well established in the context of PDE-constrained optimization (see e.g.~\cite{volker}) and have been recently also introduced in the inverse setting \cite{K16,K17}. The idea is to solve the underlying model equation simultaneously with the optimality conditions. This is in contrast to so called reduced or black-box methods, which formulate the minimization as an unconstrained optimization problem via the solution operator of \eqref{eq:abstract_modela}. The connection of the optimization approach to the Bayesian setting via the maximum a posteriori estimate is well established in the finite-dimensional setting \cite{KS10} as well as in the infinite-dimensional setting for certain prior classes, see \cite{Agapiou_2018, Clason, Dashti_2013, Helin_2015}. We refer to \cite{KS10, AMS10} for more details on the Bayesian approach to inverse problems. Recently, data-driven approaches have been introduced to the inverse setting to reduce the computational complexity in case of highly complex forward models and to improve models in case of limited understanding of the underlying process \cite{actaip}. Neural networks in the inverse framework have been successfully applied in the case of parametric holomorphic forward models \cite{HSZ20} and in the case of limited knowledge of the underlying model \cite{RPK17, RPK17_2, 2019JCoPh.378..686R, SDK2020, YMK20}. To train neural networks, gradient-based methods are typically used \cite{Goodfellow-et-al-2016}. The ensemble Kalman inversion (EKI) (see e.g.~\cite{ILS13, SS17,SS17b}) has been recently applied as gradient-free optimizer for the training of neural networks in \cite{HLR18,KS19}. 

\subsection{Our contribution}
The goal of this work is two-fold:
\begin{itemize}
\item We formulate the inverse problem in a one-shot fashion and establish the connection to the Bayesian setting. In particular, the Bayesian viewpoint allows to incorporate model uncertainty and gives a natural way to define the regularization parameters. In case of an exact forward model, the vanishing noise can be interpreted as a penalty method. This allows to establish convergence results of the one-shot formulation as an unconstrained optimization problem to the corresponding (regularized) solution of the reduced optimization problem (with exact forward model). The (numerical approximation of the) forward problem is replaced by a neural network in the one-shot formulation, i.e.~the neural network does not have to be trained in advance for all potential  parameters.
\item Secondly, we show that the EKI can be effectively used to solve the resulting optimization problem. We provide a convergence analysis in the linear setting. To enhance the performance, we modify the algorithm motivated by the continuous version of the EKI. Numerical experiments demonstrate the robustness (also in the nonlinear setting) of the proposed algorithm.
\end{itemize}
\subsection{Outline}
In section \ref{sec:pf}, we discuss the optimization and Bayesian approach to inverse problems, introduce the one-shot formulation and establish the connection to vanishing noise and penalty methods in case of exact forward models. Section \ref{sec:nn} gives an overview of neural networks used to approximate the forward problem. The EKI as a derivative-free optimizer for the one-shot formulation is discussed and analyzed in section \ref{sec:EKI}. Finally, numerical experiments demonstrating the effectiveness of the \textcolor{black}{proposed} approach are presented in section \ref{sec:ne}. We conclude in section \ref{sec:concl} with an outlook to future work.

\subsection{Notation}
We denote by $\| \cdot \|$ the Euclidean norm and by $\langle \cdot,\cdot\rangle$ the corresponding inner product. For a given symmetric, positive definite matrix (s.p.d.) $A$, the weighted norm $\|\cdot\|_A$ is defined by $\|\cdot\|_A=\|A^{-1/2}\cdot\|$ and the weighted inner product by $\langle \cdot,\cdot\rangle_A=\langle \cdot,A^{-1}\cdot\rangle$.

\section{Problem Formulation}\label{sec:pf}
 We introduce in the following the optimization and Bayesian approach to an abstract inverse problem of the form \eqref{eq:abstract_modela}-\eqref{eq:abstract_modelb}. 

Throughout this paper, we derive the methods and theoretical results under the assumption that $\mathcal X, \mathcal W$ and $\mathcal V$ are finite-dimensional, i.e.~we assume that the forward problem $M(u,p)=0$ has been discretized by a suitable numerical scheme and the parameter space is finite-dimensional as well, possibly after dimension truncation. Though most of the ideas and results can be generalized to  the infinite-dimensional setting, we avoid the technicalities  arising from the infinite-dimensional setting and focus on the discretized problem, i.e.~$\mathcal X=\mathbb R^{n_u}$, $\mathcal V=\mathbb R^{n_p}$ and $\mathcal W=\mathbb R^{n_w}$.

\subsection{Optimization approach to the inverse problem}
The optimization approach leads to the following problem
\begin{align} 
&\min_{u,p}\ \|O(p)- y\|_{\Gamma_{obs}}^2\label{eq:optipa}\\
&\quad\mbox{s.t. } M(u,p)=0\,,\label{eq:optipb}
\end{align}
i.e.~we minimize the data misfit in a suitable weighted norm with $\Gamma_{obs}\in\mathbb R^{n_y\times n_y}$ s.p.d., given that the forward model is satisfied.
Due to the ill-conditioning of the inverse problem, a regularization term on the unknown parameters is often introduced in order to stabilize the optimization, i.e.~we consider
\begin{align} 
&\min_{u,p}\|O(p)- y\|_{\Gamma_{obs}}^2 +  \alpha_1\mathcal R_1(u)\label{eq:optiprega}\\
&\quad\mbox{s.t. } M(u,p)=0\,,\label{eq:optipregb}
\end{align}
where the regularization is denoted by $\mathcal R_1:\mathcal X\to \mathbb R$ and the positive scalar $\alpha_1>0$ is usually chosen according to prior knowledge on the unknown parameter $u$. We will comment on the motivation of the regularization via the Bayesian approach in the following section. To do so, we first introduce the so-called reduced problem of \eqref{eq:optipa}-\eqref{eq:optipb} and \eqref{eq:optiprega}-\eqref{eq:optipregb}, resepectively. The forward model $M(u,p)=0$ is typically a well-posed problem, in the sense that for each parameter $u\in \mathcal X$, there exists a unique state $p \in \mathcal V$ such that $M(u,p)=0$ in $\mathcal W$. Introducing the solution operator $S:\mathcal X \to \mathcal V$ s.t.~$M(u,S(u))=0$, we can reformulate the optimization problem \eqref{eq:optipa}-\eqref{eq:optipb} as an unconstrained optimization problem
\begin{equation} \label{eq:optipred}
\min_{u \in \mathcal X}\ \|O(S(u))- y\|_{\Gamma_{obs}}^2\,,
\end{equation}
and
\begin{equation} \label{eq:optipredreg}
\min_{u \in \mathcal X}\ \|O(S(u))- y\|_{\Gamma_{obs}}^2+  \alpha_1\mathcal R_1(u)\,,
\end{equation}
respectively.

\subsection{Bayesian approach to the inverse problem}
Adopting the Bayesian approach to inverse problems, we view the unknown parameters $u$ as an $\mathcal X$-valued random variable with prior distribution $\mu_0$. The noise in the observations is assumed to be additive and described by a random variable $\eta\sim\mathcal N(0,\Gamma_{obs})$ with $\Gamma_{obs}\in\mathbb R^{n_y\times n_y}$ s.p.d, i.e.
\begin{equation}
y = O(S(u))+\eta\,,
\end{equation}
Further, we assume that the noise $\eta$ is stochastically independent of $u$. By Bayes' theorem, we obtain the posterior distribution
\begin{equation}\label{eq:posterior}
\mu^\ast(\mathrm du)\propto \exp(-\frac12\| O(S(u))-y\|_{\Gamma_{obs}}^2)\mu_0(\mathrm du)\,,
\end{equation}
 the conditional distribution of the unknown given the observation $y$.

\subsection{Connection between the optimization and the Bayesian approach}
The optimization approach leads to a point estimate of the unknown parameters, whereas the Bayesian approach computes the conditional distribution of the unknown parameters given the data, the posterior distribution, as solution of the inverse problem. Since the computation or approximation of the posterior distribution is prohibitively expensive in many applications, one often restores to point estimates such as the maximum a posteriori (MAP) estimate, the most likely point of the unknown parameters. Denoting by $\rho_0$ the Lebesgue density of the prior distribution, the MAP estimate is defined as
\begin{equation}\label{eq:MAPmax}
\argmax_{u\in\mathcal X}\ \exp\left(-\frac12\| O(S(u))-y\|_{\Gamma_{obs}}^2\right) \rho_0(u)\,.
\end{equation} 
Assuming a Gaussian prior distribution, i.e.~$\mu_0=\mathcal N(u_0,C)$, the MAP estimate is given by the solution of the following minimization problem
\begin{equation}\label{eq:MAP}
\min_{u}\ \frac12\| O(S(u))-y\|_{\Gamma_{obs}}^2 +\frac12\|u-u_0\|_C^2\,.
\end{equation}
The Gaussian prior assumption leads to a Tikhonov-type regularization in the objective function, whereas the first term in the objective function results from the Gaussian assumption on the noise. We refer the interested reader to \cite{DS17, KS10} for more details on the MAP estimate.
 
 \subsection{One-shot formulation for inverse problems}
Black-box methods are based on the reduced formulation \eqref{eq:optipred} assuming that the forward problem can be solved exactly in each iteration. Here, we will follow a different approach, which simultaneously solves the forward and optimization problem. Various names for the simultaneous solution of the design and state equation exist: one-shot method, all-at-once, piggy-back iterations etc.. We refer the reader to \cite{volker} and the references therein for more details.

Following the one-shot ideas, we seek to solve the problem
\begin{equation}\label{eq:a-a-o_forward_map}
F(u,p) = \begin{pmatrix}
M(u,p)\\ O(p)
\end{pmatrix} = \begin{pmatrix}
0\\ y
\end{pmatrix}=:\tilde y\,,
\end{equation}
Due to the noise in the observations, we rather consider
\begin{equation}\label{eq:noiseobs}
y=O(p)+\eta_{obs}
\end{equation}
with normally distributed noise $\eta_{obs}\sim\mathcal N(0,\Gamma_{obs})$, $\Gamma_{obs}\in \mathbb R^{n_y\times n_y}$ s.p.d.. Similarly, we assume that 
\begin{equation}
0=M(u,p)+\eta_{model}\,,
\end{equation}
i.e.~we assume that the model error can be described by  $\eta_{model}\sim\mathcal N(0,\Gamma_{model})$, $\Gamma_{model}\in \mathbb R^{n_w\times n_w}$ s.p.d.. Thus, we obtain the problem
 \begin{equation}\label{eq:a-a-o}
\tilde y=F(u,p) + \begin{pmatrix}
\eta_{model}\\ \eta_{obs}
\end{pmatrix}\,.
\end{equation}
The MAP estimate is then computed by the solution of the following minimization problem
\begin{equation}\label{eq:optios}
\min_{u,p}\frac12\|F(u,p)-\tilde y\|_\Gamma^2 + \alpha_1\mathcal R_1(u)+ \alpha_2\mathcal R_2(p),
\end{equation}
where $\mathcal R_1:\mathcal X\to\R$ and $\mathcal R_2:\mathcal V\to\R$ are regularizations of the parameter $u \in \mathcal X$ and the state $p \in \mathcal V$, $\alpha_1,\alpha_2>0$ and $\Gamma=\begin{pmatrix}\Gamma_{model} & 0 \\0& \Gamma_{obs}\end{pmatrix}\in\mathbb R^{(n_w+n_y) \times (n_w+n_y)}$. 
\textcolor{black}{
\begin{rmk}
Note that the proposed approach does not rely on a Gaussian noise model for the forward problem, i.e. non-Gaussian models can be straightforwardly incorporated. Then, the Bayesian viewpoint may guide the choice of the regularization parameter (or function) in the optimization approach via the MAP estimate. The model error is typically estimated from experimental data or more complex models, cp. \cite{OHagan, Higdon}. We focus here on the Gaussian setting, as the one-shot approach for inverse problems is typically formulated in a least-squares fashion (in particular, when using neural networks as surrogate models of the forward problem \cite{RPK17, RPK17_2}). The focus of this work will be on the development of a methodology, which allows to satisfy the forward problem exactly. This will be achieved by establishing the connection to the Bayesian setting and working in the vanishing noise setting.
\end{rmk}
}

\subsection{Vanishing noise and penalty methods}\label{subsec:vanishnoise}
In case of an exact forward model, i.e.~in case the forward equation is supposed to be satisfied exactly with $M(u,p)=0$, this can be modeled in the Bayesian setting by vanishing noise. In order to illustrate this idea, we consider a parametrized noise covariance model $\Gamma_{model}=\gamma \hat\Gamma_{model}$ for $\gamma\in\mathbb R_{+}$ and given s.p.d. matrix $\hat\Gamma_{model}$. The limit for $\gamma\to 0$ corresponds to the vanishing noise setting {and can be interpret as reducing the uncertainty in our model}. The MAP estimate in the one-shot framework is thus given by
\begin{equation}\label{eq:penalty}
\min_{u,p}\frac12\|O(p)-y\|_{\Gamma_{obs}}^2 +\frac{\lambda}{2} \|M(u,p)\|_{\hat\Gamma_{model}}^2 +  \alpha_1\mathcal R_1(u)+ \alpha_2\mathcal R_2(p)
\end{equation}
with $\lambda=1/\gamma$. This form reveals the close connection to penalty methods, which attempt to solve constrained optimization problems such as \eqref{eq:optipa}-\eqref{eq:optipb} by sequentially solving unconstrained optimization problems of the form \eqref{eq:penalty} for a sequence of monotonically increasing penalty parameters $\lambda$.
The following well-known result on the convergence of the resulting algorithm can be found e.g.~in \cite{bertsekas}.
\begin{prop}\label{prop:pen}
Let the observation operator $O$, the forward model $M$ and the regularization functions $\mathcal R_1$, $\mathcal R_2$ be continuous and the feasible set $\{(u,p)|M(u,p)=0\}$ be nonempty. For $k=0,1,\ldots$ let $(u_k,p_k)$ denote a global minimizer of
\begin{equation}\label{eq:penaltyseq}
\min_{u,p}\ \frac12\|O(p)-y\|_{\Gamma_{obs}}^2 +\frac{\lambda_k}{2} \|M(u,p)\|_{\hat\Gamma_{model}}^2 +  \alpha_1\mathcal R_1(u)+ \alpha_2\mathcal R_2(p)
\end{equation}
with $(\lambda_k)_{k\in\mathbb N}\subset \mathbb R_+$ strictly monotonically increasing and $\lambda_k\to\infty$ for $k\to\infty$.
Then every accumulation point of the sequence $(u_k,p_k)$ is a global minimizer of
\begin{eqnarray}\label{eq:penaltyorig}
&\min\limits_{u,p}&\frac12\|O(p)-y\|_{\Gamma_{obs}}^2 +  \alpha_1\mathcal R_1(u)+ \alpha_2\mathcal R_2(p)\\
&\mbox{s.t. }& M(u,p)=0\,.
\end{eqnarray}
\end{prop}
This classic convergence result ensures the feasibility of the estimates, i.e.~the proposed approach is able to incorporate and exactly satisfy physical constraints in the limit. We mention also the possibility to consider exact penalty terms in the objective, corresponding to different noise models in the Bayesian setting. This will be subject to future work.

This setting will be the starting point of the incorporation of neural networks into the problem. Instead of minimizing with respect to the state $p$, we will approximate the solution of the forward problem $p$ by a neural network $p_{\theta}$, where $\theta$ denote the parameters of the neural network to be learned within this framework. Thus, we obtain the corresponding minimization problem 
\begin{equation}\label{eq:min_IP_NN}
\min_{u,\theta}\ \frac12\|F(u,p_\theta)-\tilde y\|_\Gamma^2 + \alpha_1\mathcal R_1(u)+\alpha_2\mathcal R_2(p_\theta,\theta)\,,
\end{equation}
where $p_\theta$ denotes the state approximated by the neural network. 

\section{Neural Networks }\label{sec:nn}
Following, e.g.~\cite{BEG19, OPS19}, we distinguish between the neural network architecture or neural network parameters as a set of weights and biases and the neural network, which is the associated mapping when an activation function is applied to the affine linear transformations defined by the neural network parameters:

The {neural network architecture} or the {neural network parameters} $\theta$ with input dimension $d \in \N$ and $L \in \N$ layers is a sequence of matrix-vector tuples
\begin{align*}
\theta &= \big((W_{\ell},b_{\ell})\big)_{\ell = 1}^L = \big(W_1,b_1),(W_2,b_2),\ldots,(W_L,b_L)\big) \in \Theta\,,
\end{align*}
where $\Theta := \times_{\ell=1}^{L}(\R^{N_\ell \times N_{\ell-1}}\times \R^{N_\ell}) \cong \mathbb R^{n_\theta}$, with $n_\theta=\sum_{\ell = 1}^L (N_{\ell-1}+1)\,N_{\ell}$. Hence, the number of neurons in layer $\ell$ is given by $N_{\ell} \in \N$ and $N_0 := d$, i.e.~$W_{\ell} \in \R^{N_{\ell}\times N_{\ell-1}}$ and $b_{\ell} \in \R^{N_{\ell}}$ for $\ell = 1,\ldots,L$.
Given an activation function $\varrho: \R \to \R$, we call the mapping 
\begin{align*}
p^{\varrho}_\theta:\,&\bbR^d \to \bbR^{N_L}\,\\
&x \mapsto p^{\varrho}_\theta(x):=x_L\,,
\end{align*}
defined by the recursion
\begin{align*}
x_0 &:= x\,,\\
x_{\ell} &:= \varrho(W_{\ell}x_{\ell-1} + b_{\ell})\,, \quad \text{for } \ell = 1,\ldots,L-1\,,\\
x_L &:= W_{L}x_{L-1} + b_{L}\,,
\end{align*}
a {neural network} $p_\theta^\varrho$ with activation function $\varrho$, where $\varrho$ is understood to act component-wise on vector-valued inputs, i.e.~$\varrho(z) := \big(\varrho(z_1),\ldots,\varrho(z_n)\big)$ for $z \in \R^n$. 
In the following the activation function is always the sigmoid function $\varrho = \frac{1}{1 + e^{-x}}$. We will therefore waive the superscript indicating the activation function, i.e.~we abbreviate $p_\theta^\varrho = p_\theta$. 

This class of neural networks is often called feed-forward deep neural networks (DNNs) and have recently found success in forward \cite{SZ19} as well as Bayesian inverse \cite{HSZ20} problems in the field of uncertainty quantification. 

Based on the approximation results of polynomials by feed-forward DNNs \cite{Y17}, the authors in \cite{SZ19} derive bounds on the expression rate for \textcolor{black}{multivariate}, real-valued functions depending holomorphically on a sequence $z = (z_j)_{j\in \mathbb N}$ of parameters. More specifically, the authors consider functions that admit sparse Taylor \textcolor{black}{generalized polynomial chaos (gpc)} expansions, i.e.~$s$-summable Taylor gpc coefficients. Such functions arise as response surfaces of parametric PDEs, or in a more general setting from parametric operator equations, see e.g.~\cite{S12} and the references therein. Their main results is that these functions can be expressed with arbitrary accuracy $\delta > 0$ (uniform w.r.~to $z$) by DNNs of size bounded by $C\delta^{-s/(1-s)}$ with a constant $C > 0$ independent of the dimension of the input data $z$. Similar results for parametric PDEs can be found in \cite{KPRS19}.

The methods in \cite{SZ19} motivated the work of \cite{HSZ20} in which the authors show holomorphy of the data-to-QoI map $y \mapsto \mathbb{E}^{\mu^*}[\text{QoI}]$\textcolor{black}{, which relates observation data to the posterior expectation of an unknown quantity of interest (QoI),} for additive, centered Gaussian observation noise in Bayesian inverse problems. Using the fact that holomorphy implies fast convergence of Taylor expansions, the authors derived an exponential expression rate bound in terms of the overall network size. 

Our approach differs from the ideas above as we do not want to approximate the data-to-QoI map, but instead emulate the state $p$ itself by a DNN. Hence, in our method the input of the neural network is a point in the domain of the state, $x\in D$. The output of the neural network is \textcolor{black}{an approximation of} the state at this point, $p_\theta(x) \in \R$, i.e.~$N_L = 1$. 
\textcolor{black}{By a slight abuse of notation we denote by $p_\theta \in \mathcal{V} =\R^{n_p}$ also a vector containing evaluations of the neural network at the $n_p$-many grid points of the state.}
In combination with a one-shot approach for the training of the neural network parameters, our method is closer related to the physics-informed neural networks (PINNs) in \cite{RPK17, RPK17_2}.

In \cite{RPK17, RPK17_2} the authors consider PDEs of the form
\begin{align*}
    f(t,x) := p_t + N(p,\lambda) = 0, \quad t \in [0,T],\,x\in D\,,
\end{align*}
where $N$ is a nonlinear differential operator parametrized by $\lambda$. The authors replace $p$ by a neural network $p_\theta$ and use automatic differentiation to construct the function $f_\theta(t,x)$. The neural network parameters are then obtained by minimizing $\text{MSE} = \text{MSE}_p + \text{MSE}_f$, where 
\begin{align*}
\text{MSE}_p := \frac{1}{N_p} \sum_{i=1}^{N_p} |p_{\textcolor{black}{\theta}}(t_p^i,x_p^i) - p^i|^2\,, \quad \text{MSE}_f := \frac{1}{N_f} \sum_{i=1}^{N_f} |f_{\textcolor{black}{\theta}}(t_f^i,x_f^i)|^2\,,
\end{align*}
\textcolor{black}{and $\{t_p^i,x_p^i,p^i\}_{i=1}^{N_p}$ denote the training data and $\{t_f^i,x_f^i\}_{i=1}^{N_f}$ are collocation points of $f_\theta(t,x)$.}
For the minimization a L-BFGS method is used. The parameters $\lambda$ of the differential operator turn into parameters of the neural network $f_\theta$ and can be learned by minimizing the MSE.

In \cite{YMK20} the authors consider so called Bayesian neural networks (BNNs), where the neural network parameters are updated according to Bayes' theorem. Hereby the initial distribution on the network parameters serves as prior distribution. The likelihood requires the PDE solution, which is obtained by concatenating the Bayesian neural network with a physics-informed neural network, which they call Bayesian physics-informed neural networks (B-PINNs). For the estimation of the posterior distributions they use the Hamiltonian Monte Carlo method and variational inference. In contrast to the PINNs, the Bayesian framework allows them to quantify the aleatoric uncertainty associated with noisy data. In addition to that, their numerical experiments indicate that B-PINNs beat PINNs in case of large noise levels on the observations.
In contrast to that, our proposed method is based on the MAP estimate and remains exact in the small noise limit. We propose a derivative-free optimization method, the EKI, which shows promising results (also compared to quasi-Newton methods) without requiring derivatives w.r.~to the weights and design parameters.


\section{Ensemble Kalman Filter for Neural Network Based One-shot Inversion}\label{sec:EKI}

%

The ensemble Kalman inversion (EKI) generalizes the well-known ensemble \linebreak[4]Kalman Filter (EnKF) introduced by Evensen and coworker in the data assimilation context \cite{GE03} to the inverse setting, see \cite{ILS13} for more details. Since the Kalman filter involves a Gaussian approximation of the underlying posterior distribution, we focus on an iterative version based on tempering in order to reduce the linearization error. Recall the posterior distribution $\mu^\ast$ given by
\begin{equation*}
\mu^\ast(\mathrm dv)\propto \exp(-\frac12\| G(v)-y\|_{\Gamma}^2)\mu_0(\mathrm dv)\,.
\end{equation*}
for an abstract inverse problem 
\begin{equation*}
y=G(v)+\eta
\end{equation*}
with $G$ mapping from the unknowns $v\in \mathbb R^{n_v}$ to the observations $y\in\mathbb R^{n_y}$ with $\eta\sim\mathcal N(0,\Gamma)$, $\Gamma\in \mathbb R^{n_y\times n_y}$ s.p.d.. We define the intermediate measures
\begin{equation}\label{eq:SMCposterior}
\mu_n(\mathrm dv)\propto \exp(-\frac12 nh\| G(v)-y\|_{\Gamma}^2)\mu_0(\mathrm dv)\, \quad n=0,\ldots,N
\end{equation}
by scaling the data misfit / likelihood by the step size $h=N^{-1}$, $N\in\mathbb N$. The idea is to evolve the prior distribution $\mu_0$ into the posterior distribution $\mu_N=\mu^\ast$ by this sequence of intermediate measures and to apply the EnKF to the resulting artificial time dynamical system. Note that we account for the repeated use of the observations by amplifying the noise variance by $N=1/h$ in each step. The EKI then uses an ensemble of $J$ particles $\{v^{(j)}_0\}_{j=1}^J$ with $J\in \mathbb N$ to approximate the intermediate measures $\mu_n$ by
\begin{equation}
  {\mu_{n}}\simeq \frac 1J \sum_{j=1}^J \delta_{v_n}^{(j)}
  \end{equation}
  with $\delta_{v}$ denoting the \textcolor{black}{Dirac measure centered on} $v_n^{(j)}$. The particles are transformed in each iteration by the application of the Kalman update formulas to the empirical mean $\bar v_n=\frac1J\sum_{j=1}^Jv_n^{(j)}$ and covariance $C(v_n)=\frac{1}{J-1}\sum_{j=1}^J (v_n^{(j)}-\bar v_n)\otimes (v_n^{(j)}-\bar v_n)$ in the form
  \begin{equation}
\bar v_{n+1}=\bar v_{n}+K_n(y-  G(\bar v_n)))\,, \qquad C(v_{n+1})=C(v_n)-K_nC^{y,v}(v_n)\,,
  \end{equation}
  where $K_n=C^{v,y}(\textcolor{black}{v_n})(C^{y,y}(\textcolor{black}{v_n})+\frac 1h\Gamma)^{-1}$ denotes the Kalman gain
  and, for $v=\{v^{(j)}\}_{j=1}^J$, the operators 
$C^{yy}$ and $C^{vy}$ given by
 \begin{align}
  \qquad C^{y,y}(v)&=\frac{1}{J}\sum_{j=1}^J\bigl( G(v^{(j)})-\bar G\bigr)
\otimes\bigl(G(v^{(j)})-\bar G\bigr), \label{eq:sa}\\
  \qquad C^{v,y}(v)&=\frac{1}{J}\sum_{j=1}^J\bigl(v^{(j)}-\bar v\bigr)
\otimes\bigl(G(v^{(j)})-\bar G\bigr), \label{eq:sb}\\
\qquad C^{y,v}(v)&=\frac{1}{J}\sum_{j=1}^J\bigl(G(v^{(j)})-\bar G\bigr)
\otimes\bigl(v^{(j)}-\bar v\bigr), \label{eq:sc}\\
\bar G&=\frac1J\sum_{j=1}^J
 G(v^{(j)})\label{eq:sd}
  \end{align}
are the empirical covariances and empirical mean in the observation space. Since this update does not uniquely define the transformation of each particle $v_n^{(j)}$ to the next iteration $v_{n+1}^{(j)}$, the specific choice of the transformation leads to different variants of the EKI. We will focus here on the generalization of the EnKF as introduced by \cite{ILS13} resulting in a mapping of the particles of the form
   \begin{equation} \label{eq:EnKFiter}
 v_{n+1}^{(j)}=v_{n}^{(j)}+C^{v,y}(v_n)(C^{y,y}(v_n)+h^{-1}\Gamma)^{-1}
\bigl(y_{n+1}^{(j)}-G(v_n^{(j)})\bigr), \quad j=1,\cdots, J,
 \end{equation}
 where
$$y_{n+1}^{(j)}=y+\xi^{(j)}_{n+1}$$
The $\xi^{(j)}_{n+1}$ are i.i.d.~random variables
distributed according to $\mathcal N(0,h^{-1}\Sigma)$
with $\Sigma=\Gamma$ corresponding to the case of perturbed observations
and $\Sigma=0$ {to the unperturbed observations}.

The motivation via the sequence of intermediate measures and the resulting artificial time allows to derive the continuous time limit of the iteration, which has been extensively studied in \cite{BSWW19, SS17, SS17b} to build analysis of the EKI in the linear setting. 
This limit arises by taking the parameter $h$ in \eqref{eq:EnKFiter} to zero resulting in
 \begin{eqnarray}\label{eq:ode}
   \frac{\dd v^{(j)}}{\dd t}&=&C^{v,y}(v)\Gamma^{-1}
(y-G(v^{(j)}))+C^{v,y}(v^{(j)})\Gamma^{-1}\sqrt{\Sigma} \,
\frac{dW^{(j)}}{dt}.
\end{eqnarray}
As shown in \cite{doi:10.1137/140981319}, the EKI does not in general converge to the true posterior distribution. Therefore, the analysis presented in \cite{BSWW19, SS17, SS17b} views the EKI as a derivative-free optimizer of the data misfit, which is also the viewpoint we adopt here.

\subsubsection{Ensemble Kalman inversion for neural network based one-shot formulation}

By approximating the state of the underlying PDE by a neural network, we seek to optimize the unknown parameter $u$ and on the other side the parameters of the neural network $\theta$. The idea is based on defining the function $H(v):=H(u,\theta)= F(u,p_{\theta})$, where $p_\theta$ denotes the state approximated by the neural network and $v=(u,\theta)^\top$. This leads to the empirical summary statistics
\begin{equation*}
\overline{(u,\theta)}_n = \frac1J\sum_{j=1}^J (u_n^{(j)},\theta_n^{(j)}),\quad \bar{H}_n = \frac{1}{J}\sum^{J}_{j=1}H(u_n^{(j)},\theta_n^{(j)}), 
\end{equation*}
\begin{align*}
C^{u\theta,y}_{n} &= \frac{1}{J}\sum^{J}_{j=1} \bigl((u_n^{(j)},\theta_n^{(j)})^\top - \overline{(u,\theta)}^\top_n\bigr)\otimes ({H}\bigl(u_n^{(j)},\theta_n^{(j)}) - \bar{{H}}_n\bigr), \\ C^{y,y}_{n} &= \frac{1}{J}\sum^{J}_{j=1}
\bigl({H}(u_n^{(j)},\theta_n^{(j)}) - \bar{{H}}_n\bigr)
\otimes   \bigl({H}(u_n^{(j)},\theta_n^{(j)}) - \bar{{H}}_n\bigr),
\end{align*}
and the EKI update 
\begin{equation}
\label{eq:EKI_PINNS}
(u_{n+1}^{(j)},\theta_{n+1}^{(j)})^\top= (u_n^{(j)},\theta_n^{(j)})^\top + C^{u\theta,y}_{n} \big(C^{y,y}_n +  h^{-1}\Gamma\big)^{-1}\big(\tilde y^{(j)}_{n+1} - {H}(u^{(j)}_n,\theta_n^{(j)})\big),
\end{equation}
where the perturbed observation are computed as before
\begin{equation}
\label{eq:dataU_PINNS}
\tilde y_{n+1}^{(j)}=\tilde y+\xi^{(j)}_{n+1}, \quad \xi^{(j)}_{n+1} \sim 
\mathcal N(0,h^{-1}\Sigma),
\end{equation}
with $$\tilde y = \begin{pmatrix}
0 \\ y
\end{pmatrix}, \quad \Gamma := \begin{pmatrix}
\Gamma_{model} & 0\\ 0 & \Gamma_{obs}
\end{pmatrix}.$$

{Figure \ref{fig:EKI_PINNs} illustrates the basic idea of the application of the EKI to solve the neural network based one-shot formulation.}

\begin{figure}[!htb]
	\centering
	\includegraphics[width=1\textwidth]{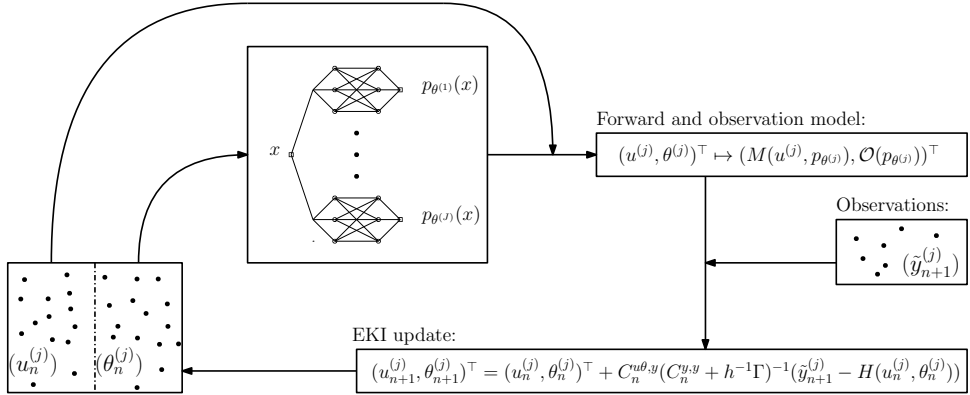}
    \caption{Description of the EKI applied to solve the neural network based one-shot formulation.}\label{fig:EKI_PINNs}
\end{figure}

The EKI \eqref{eq:EKI_PINNS} will be used as a derivative-free optimizer of the data misfit
$
\|F(u,p_{\theta})-\tilde y\|_\Gamma^2\,.
$
The analysis presented in \cite{BSWW19, SS17, SS17b} shows that the EKI in its continuous form is able to recover the data with a finite number of particles in the limit $t\to \infty$ under suitable assumptions on the forward problem and the set of particles. In particular, the analysis assumes a linear forward problem. Extensions to the nonlinear setting can be found e.g.~in \cite{chada2019convergence}. The limit $t\to \infty$ corresponds to the noise-free setting, as the inverse noise covariance scales with $n/N=nh$ in \eqref{eq:SMCposterior}. To explore the scaling of the noise and to discuss regularization techniques, we illustrate the ideas in the following for a Gaussian linear setting, i.e.~we assume that the forward response operator is linear $H(v)=Av$ with $A\in\mathcal L(\mathcal X \times {\Theta}, \mathbb R^{n_w+n_y})$ and $\mu_0=\mathcal N(v_0,C_0)$. Considering the large ensemble size limit $J\to \infty$, the mean $m$ and covariance $C$ satisfy the equations
 \begin{eqnarray}\label{eq:odelinearmC}
   \frac{\dd m(t)}{\dd t}&=&-C(t)A^\top \Gamma^{-1} (Am(t)-y)\\
    \frac{\dd C}{\dd t}&=&-C(t)A^\top \Gamma^{-1} A C(t)
\end{eqnarray}
for $\Sigma=\Gamma$ in \eqref{eq:EnKFiter}. By considering the dynamics of the inverse covariance, it is straightforward to show that the solution is given by
 \begin{eqnarray}\label{eq:odelinearCinv}
      C^{-1}(t)&=&C_0^{-1}+A^\top\Gamma^{-1}At\,,
\end{eqnarray}
see e.g.~\cite{garbunoinigo2019interacting} and the references therein for details.
Note that $C(1)$ corresponds to the posterior covariance and that $C(t)\to 0$ for $t\to\infty$. Furthermore, the mean is given by
 \begin{eqnarray}\label{eq:odelinear_mean}
      m(t)&=&(C_0^{-1}+A^\top\Gamma^{-1}At)^{-1}(A^\top \Gamma^{-1}yt+C_0^{-1}v_0)\,,
\end{eqnarray}
in particular the mean minimizes the data misfit in the limit $t\to \infty$. The application of the EKI in the inverse setting therefore often requires additional techniques such as adaptive stopping \cite{SS17b} or additional regularization \cite{CST19} to overcome the ill-posedness of the minimization problem. 
To control the regularization of the data misfit and neural network individually, we consider the following system
 \begin{align}
F(u,p_\theta) + \begin{pmatrix}
\eta_{model}\\ \eta_{obs}
\end{pmatrix} &= \tilde y\\
\begin{pmatrix}
u\\ \theta
\end{pmatrix} + \begin{pmatrix}
\eta_{param}\\ \eta_{NN}
\end{pmatrix} &= 0
\end{align}
with $\eta_{model}\sim\mathcal N(0,1/\lambda\, \hat \Gamma_{model})$, $\eta_{obs}\sim\mathcal N(0,\Gamma_{obs})$, $u\sim \mathcal N(u_0,1/\alpha_1\, C)$ and $\theta\sim \textcolor{black}{\mathcal N(0,1/\alpha_2\, I)}$. The loss function for the augmented system is therefore given by
\begin{equation}\label{eq:lossaugmented}
\frac12\|O(p_\theta)-y\|_{\Gamma_{obs}}^2 +\frac{\lambda}{2} \|M(u,p_\theta)\|_{\hat\Gamma_{model}}^2 +  \textcolor{black}{\frac{\alpha_1}{2}}\mathcal \|u-u_0\|^2_C+ \textcolor{black}{\frac{\alpha_2}{2}}\mathcal \|\theta\|^2\,.
\end{equation}
Assuming that the resulting forward operator
\begin{equation}\label{eq:augsys}
G(u,\theta)=\begin{pmatrix}F(u,p_\theta)\\u\\ \theta \end{pmatrix}
\end{equation}
is linear, the EKI will converge to the minimum of the regularized loss function \eqref{eq:lossaugmented}, cp.~\cite{SS17}.
To ensure the feasibility of the EKI estimate (w.r.~to the underlying forward problem), we propose the following algorithm using the ideas discussed in section \ref{subsec:vanishnoise}.
\setcounter{figure}{0} 
\begin{algorithm}
	\label{alg:pEKI}
\begin{algorithmic}[1]
\Require initial ensemble $v_0^{(j)}=(u_0^{(j)},\theta_0^{(j)})^\top\in \mathcal X \times {\Theta}, j=1,\ldots J,\ \lambda_0$.
\For{$k=0,1,2,\dots$}
  \State Compute an approximation of the minimizer $(u_k,\theta_k)^\top$ of
  \begin{equation*}
\min_{u,\theta}\ \frac12\|O(p_\theta)-y\|_{\Gamma_{obs}}^2 +\frac{\lambda_k}{2} \|M(u,p_\theta)\|_{\hat\Gamma_{model}}^2 + \textcolor{black}{\frac{\alpha_1}{2}}\mathcal \|u-u_0\|^2_C+ \textcolor{black}{\frac{\alpha_2}{2}}\mathcal \|\theta\|^2\,.
\end{equation*}
  by solving
  \begin{eqnarray*}
   \frac{\dd v^{(j)}}{\dd t}&=&C^{vy}(v)\Gamma^{-1}
(\textcolor{black}{\hat{y}}- G(v^{(j)}))+C^{vy}(v^{(j)})\Gamma^{-1}\sqrt{\Sigma}\,
\frac{dW^{(j)}}{dt}
\end{eqnarray*}
with \textcolor{black}{$\hat{y} = (0,y,0,0)^\top$,} $v^{(j)}(0)=v^{(j)}_0$ for the system \eqref{eq:augsys} and {\footnotesize \textcolor{black}{$\Gamma = 
\text{diag}\left(C,I,\tfrac{1}{\alpha_1}C,\tfrac{1}{\alpha_2}I\right)$.}}
\State Set $v_k=(u_k,\theta_k)^\top=\textcolor{black}{\lim_{T\to \infty} \bar v(T)}$.
  \State Increase $\lambda_k$.
  \State Draw $J$ ensemble members $v_0^{(j)}$ from $\mathcal N(v_k, \scriptsize{\begin{pmatrix}  C & 0\\0 &  I\end{pmatrix}})$.
\EndFor
\end{algorithmic}
\renewcommand{\figurename}{Alg.}
\captionof{figure}{Penalty ensemble Kalman inversion for neural network based one-shot inversion}
\renewcommand{\figurename}{Fig.}
\end{algorithm} 
\renewcommand\thefigure{\arabic{figure}}

\begin{thm}
\label{theo:conv}
Assume that the forward operator {$G: \mathcal X\times \Theta \to \mathbb R^{n_G}$,} $n_G:=n_w+n_y+n_u+n_\theta$,
\begin{equation*}
 G(u,\theta)=\begin{pmatrix}F(u,p_\theta)\\u\\ \theta \end{pmatrix}
\end{equation*}
is linear, i.e.~$F(u,p_\theta)=A\begin{pmatrix}u \\ \theta\end{pmatrix}$ with $A\in{\mathcal L(\mathcal X\times \Theta, \mathbb R^{n_w+ n_y})}$. Let $(\lambda_k)_{k\in\mathbb N}\subset \mathbb R_+$ be strictly monotonically increasing and $\lambda_k\to\infty$ for $k\to\infty$.
Further, assume that the initial ensemble members are chosen so that {$\spann\{(u^{(j)}(0),\theta^{(j)}(0))^\top, j=1,\ldots,J\}=\mathcal X\times \Theta$}.

Then, Algorithm 1 generates a sequence of estimates $(\bar u_k, \bar \theta_k)_{k\in\mathbb N}$, where $\bar u_k, \bar \theta_k$ minimizes the loss function for the augmented system given by
\begin{equation*}
\frac12\|O(p_\theta)-y\|_{\Gamma_{obs}}^2 +\frac{\lambda_k}{2} \|M(u,p_\theta)\|_{\hat\Gamma_{model}}^2 +  \textcolor{black}{\frac{\alpha_1}{2}}\mathcal \|u-u_0\|^2_C+ \textcolor{black}{\frac{\alpha_2}{2}}\mathcal \|\theta\|^2\,
\end{equation*}
with given $\alpha_1, \alpha_2>0$.
Furthermore, every accumulation point of $(\bar u_k, \bar \theta_k)_{k\in\mathbb N}$ is the (unique, global) minimizer of
\begin{eqnarray*}
&\min\limits_{u,\theta}&\frac12\|O(p_\theta)-y\|_{\Gamma_{obs}}^2 + \textcolor{black}{\frac{\alpha_1}{2}}\mathcal \|u-u_0\|^2_C+ \textcolor{black}{\frac{\alpha_2}{2}}\mathcal \|\theta\|^2\\
&\mbox{s.t. }& M(u,p_\theta)=0\,
\end{eqnarray*}
\end{thm}
\begin{proof}
Under the assumption of a linear forward model, the penalty function 
\begin{equation*}
\frac12\|O(p_\theta)-y\|_{\Gamma_{obs}}^2 +\frac{\lambda_k}{2} \|M(u,p_\theta)\|_{\hat\Gamma_{model}}^2 +  \textcolor{black}{\frac{\alpha_1}{2}}\mathcal \|u-u_0\|^2_C+ \textcolor{black}{\frac{\alpha_2}{2}}\mathcal \|\theta\|^2\,
\end{equation*}
is strictly convex for all $k\in\mathbb N$, i.e.~there exists a unique minimizer of the penalized problem. Choosing the initial ensemble such that  {$\spann\{(u^{(j)}(0),\theta^{(j)}(0))^\top, j=1,\ldots,J\}=\mathcal X\times \Theta$} ensures the convergence of the EKI estimate to the global minimizer, see \cite[Theorem 3.13]{CST19} and \cite[Theorem 4]{SS17}. The convergence of Algorithm 1 to the minimzer of the constrained problem then follows from Proposition \ref{prop:pen}. 
\end{proof}
\begin{rmk}
Note that the convergence result Theorem \ref{theo:conv} involves an assumption on the size of the ensemble to ensure the convergence to the (global) minimizer of the loss function in each iteration. This is due to the well-known subspace property of the EKI, i.e.~the EKI estimate will lie in the span of the initial ensemble when using the EKI in its variant as discussed here. In case of a large or possibly infinite-dimensional parameter / state space, the assumption on the size of the ensemble is usually not satisfied in practice. Techniques such as variance inflation, localization and adaptive ensemble choice are able to overcome the subspace property and thus might lead to much more efficient algorithms from a computational point of view. Furthermore, we stress the fact that the convergence result presented above requires the linearity of the forward and observation operator (w.r.~to the optimization variables), i.e.~the assumption is not fulfilled when considering neural networks with a nonlinear activation function as approximation of the forward problem. \textcolor{black}{Please note that the proposed approach can be applied in a rather general framework for arbitrary approximation methods of the forward problem, i.e. in particular for finite element approximations. However, in the case of neural networks with nonlinear activation functions, the assumptions are not met, but numerical experiments show promising results even in the nonlinear setting. Therefore, Theorem 1 is included to cover the case with linear approximations of the forward problem and as a starting point for the analysis of nonlinear problems, which  will be subject to future work. }
\end{rmk}

\textcolor{black}{To accelerate the computation of the minimizer, we suggest the following variant. Algorithm 1 requires the solution of a sequence of optimization problems, i.e. for each $\lambda$, EKI is used to approximate the solution of the corresponding minimization problem. To avoid the repeated application of EKI, the idea of Algorithm 2 is to solve just one optimization problem with increasing regularization parameter $\lambda$ (instead of the sequence of optimization problems). This is incorporated in the continuous version of EKI by solving an additional differential equation for $\lambda$ with nondecreasing right-hand-side. The computational effort is thus reduced and numerical experiments suggest a comparable performance in terms of accuracy. The theoretical analysis of the convergence behavior will be subject to future work.} 
\begin{algorithm}
	\label{alg:pEKI2}
\begin{algorithmic}[1]
\Require initial ensemble $v_0^{(j)}=(u_0^{(j)},\theta_0^{(j)})^\top\in \mathcal X \times {\Theta}, j=1,\ldots, J,\ \lambda_0\in \mathbb R_{\ge 0}$, $f:\mathbb R_{\ge 0}\to\mathbb R_+ $.
  \State Compute an approximation of the minimizer of
  \begin{eqnarray*}
&\min\limits_{u,\theta}&\frac12\|O(p_\theta)-y\|_{\Gamma_{obs}}^2 + \textcolor{black}{\frac{\alpha_1}{2}}\mathcal \|u-u_0\|^2_C+ \textcolor{black}{\frac{\alpha_2}{2}}\mathcal \|\theta\|^2\\
&\mbox{s.t. }& M(u,p_\theta)=0\,
\end{eqnarray*}
  by solving the following system
  \begin{eqnarray*}
   \frac{\dd v^{(j)}}{\dd t}&=&C^{vy}(v)\Gamma^{-1}
(\textcolor{black}{\hat{y}}- G(v^{(j)}))+C^{vy}(v^{(j)})\Gamma^{-1}\sqrt{\Sigma} \,
\frac{dW^{(j)}}{dt}\\
 \frac{\dd \lambda}{\dd t}&=&f(\lambda)
\end{eqnarray*}
with  \textcolor{black}{$\hat{y} = (0,y,0,0)^\top$,} $v^{(j)}(0)=v^{(j)}_0$ for the system \eqref{eq:augsys}, $\lambda(0)=\lambda_0$ and {\footnotesize \textcolor{black}{$\Gamma = \text{diag}\left(C,I,\tfrac{1}{\alpha_1}C,\tfrac{1}{\alpha_2}I\right)$}}.
\end{algorithmic}
\renewcommand{\figurename}{Alg.}
\captionof{figure}{Simultaneous penalty ensemble Kalman inversion for neural network based one-shot inversion}
\renewcommand{\figurename}{Fig.}
\end{algorithm} 
\renewcommand\thefigure{\arabic{figure}}
\setcounter{figure}{1}

\section{Numerical Results}\label{sec:ne}
We present in the following numerical experiments illustrating the one-shot inversion. \textcolor{black}{The first example is a one-dimensional problem, for which we compare the black-box method, quasi-Newton method for the one-shot inversion, {quasi-Newton method for the neural network based one-shot inversion (Algorithm 1),} EKI for the one-shot inversion and  EKI for the neural networks based one-shot inversion (Algorithm 2) in the linear setting.} \textcolor{black}{Furthermore, we numerically explore the convergence behavior of the EKI for the neural networks based one-shot inversion (Algorithm 2) also for a nonlinear forward model. The second experiment is concerned with the extension of the linear model to the two-dimensional problem to investigate the potential of the EKI for neural network based inversion in the higher-dimensional setting.}
\subsection{One-dimensional example}

We consider the problem of recovering the unknown data $u^\dagger$ from noisy observations $$y = \mathcal O(p^\dagger)+\eta^\dagger,$$
where $p^\dagger =\mathcal{A}^{-1}(u^\dagger)$ is the solution of the one-dimensional elliptic equation
\begin{equation}\label{FP}
\begin{split}
-\frac{\mathrm{d}^2p}{\mathrm{d} x^2}+p&=u^\dagger \quad \mbox{in} \ D:=(0,\pi), \cr
	p&=0 \quad \mbox{on}\ \partial D,
\end{split}
\end{equation}
with operator $\cO$ observing the dynamical system at $n_y=2^3-1$ equispaced observation points $x_i=\frac i{2^4}\cdot \pi$, $i=1,\dots,n_y$. 

We approximate the forward-problem \eqref{FP} numerically on a uniform mesh with meshwidth $h=2^{-6}$ by a finite element method with continuous, piecewise linear ansatz functions. The approximated solution operator will be denoted by $S\in\R^{n_u\times n_u}$, with $n_u=1/h$.

{The unknown parameter $u$ is assumed to be Gaussian, i.e.~$u\sim\mathcal N(0,C_0)$, with (discretized) covariance operator $C_0=\beta(-\frac{\mathrm{d}^2}{\mathrm{d} x^2})^{-\nu}$ for $\beta=5, \ \nu=1.5$.
For our inverse problem we assume a observational noise covariance $\Gamma_{obs}= 0.1\cdot I_{n_y}$, a model error covariance $\hat\Gamma_{model} = 100\cdot I_{n_u}$ and we choose the regularization parameter $\alpha_1 = 0.002$, while we turn off the regularization on $p$, i.e.~we set $\alpha_2=0$.} Further, we choose a feed-forward DNN with $L=3$ layers, where we set $N_1=N_2=10$ size of the hidden layers and $N_0= N_L=1$ size of the input and output layer. As activation function we choose the sigmoid function $\varrho(x) = \frac{1}{1 + e^{-x}}$. 
The EKI method is based on the deterministic formulation represented through the coupled ODE system
\begin{equation}\label{eq:EKI_ode}
\frac{\dd v^{(j)}}{\\d t} = C^{vy}(v)\Gamma^{-1}(y-G(v^{(j)})),
\end{equation}
which will be solved with the \texttt{MATLAB} function \texttt{ode45} up to time $T=10^{10}$. {The ensemble of particles $(u^{(j)})$, $(u^{(j)},p^{(j)})$ and $(u^{(j)},\theta^{(j)})$ respectively will be initialized by $J=150$ particles as i.i.d.~samples, where the parameters $u_0^{(j)}$ are drawn from the prior distribution $\mathcal N(0,C_0)$, the state $p_0^{(j)}$ are drawn from $\mathcal N(0,5\, I_{n_p})$ and the weights of the neural network approximation are drawn from $\mathcal N(0,I_{n_\theta})$, which are all independent from each other. 
}
 
We compare the results to a classical gradient-based method, namely the quasi-Newton method with BFGS updates,  as implemented by \texttt{MATLAB}.

We summarize the methods in the following and introduce abbreviations:
\begin{enumerate}
\item reduced formulation: explicit solution (redTik).
\item one-shot formulation: we compare the performance of the EKI with Algorithm 1 (osEKI\_1),  the EKI with Algorithm 2 (osEKI\_2) and the quasi-Newton method with Algorithm 1 (osQN\_1).
\item neural network based one-shot formulation: we compare the performance of the EKI with Algorithm 2 (nnosEKI\_2) and the quasi-Newton method with Algorithm 1 (nnosQN\_1).
\end{enumerate}

Figure \ref{fig:ex1_scaling1} shows the increasing sequence of $\lambda$ used for Algorithm 1 and the quasi-Newton method and Algorithm 2 (over time).
\begin{figure}[!htb]
	\centering
	\includegraphics[width=.5\textwidth]{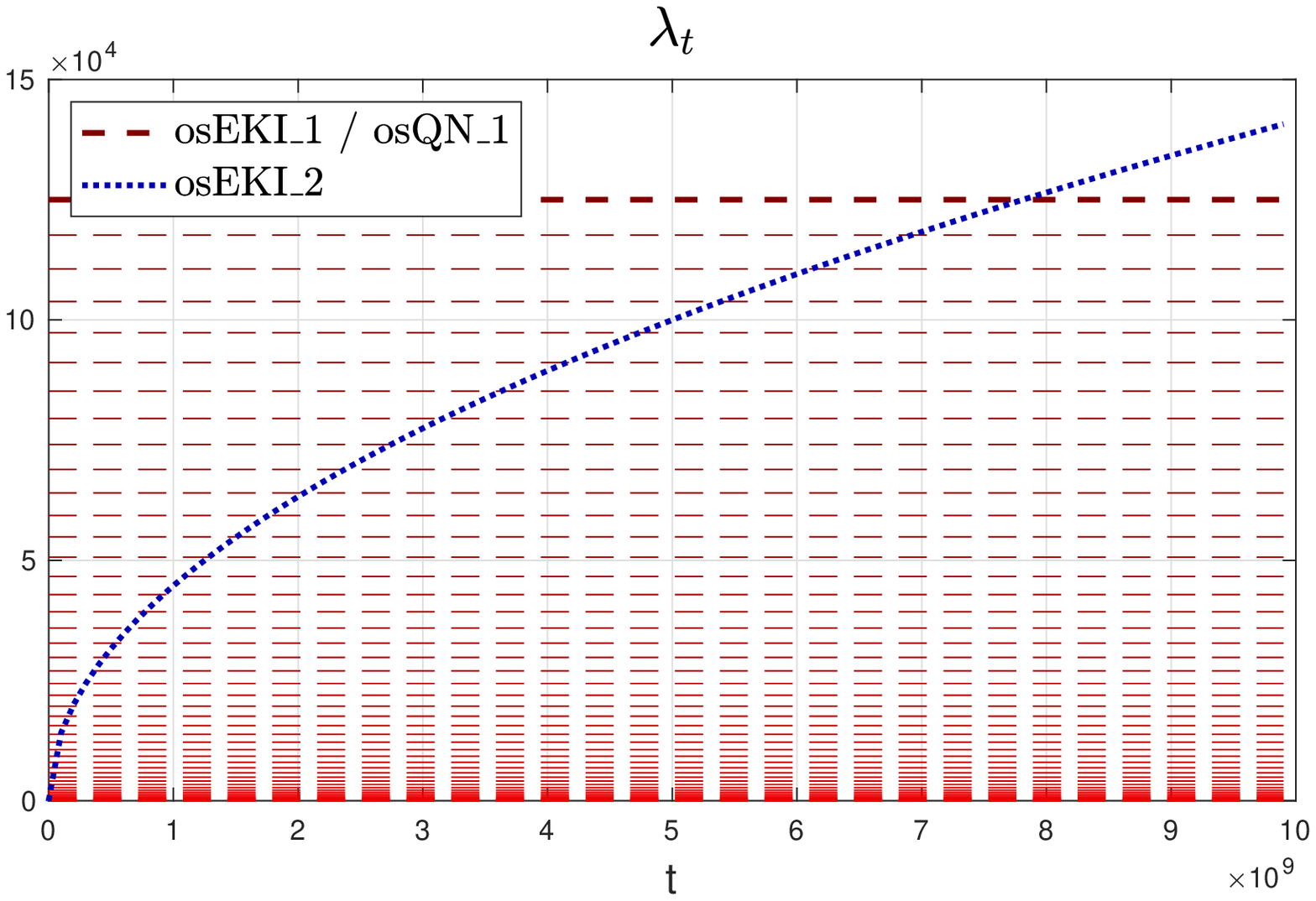}
    \caption{Scaling parameter $\lambda$ depending on time for Algorithm 1, {$\lambda_k= k^3$ for $k=1,2,\ldots,50$}, and Algorithm 2, {$\mathrm d\lambda/\mathrm dt=1/\lambda$}.}\label{fig:ex1_scaling1}
\end{figure}

\subsubsection{One-shot inversion}
To illustrate the convergence result of the EKI and to numerically investigate the performance of Algorithm 2, we start the discussion by a comparison of the one-shot inversion based on the FEM approximation of the forward problem in the 1d example.

Figure \ref{fig:ex1_est} illustrates the difference of the estimates given by EKI with Algorithm
1 (osEKI\_1), the EKI with Algorithm 2 (osEKI\_2) and the quasi-Newton
method with Algorithm 1 (osQN\_1) compared to the Tikhonov solution and the truth (on the left-hand side) and in the observation space (on the right-hand side). We observe that all three methods lead to an excellent approximation of the Tikhonov solution. Due to the linearity of the forward problem, the quasi-Newton method as well as the EKI with Algorithm 1 are expected to converge to the regularized solution. The EKI with Algorithm 2 shows a similar performance while reducing the compuational effort significantly compared to Algorithm 1.
\begin{figure}[!htb]
	\begin{subfigure}[c]{0.49\textwidth}
	\includegraphics[width=1\textwidth]{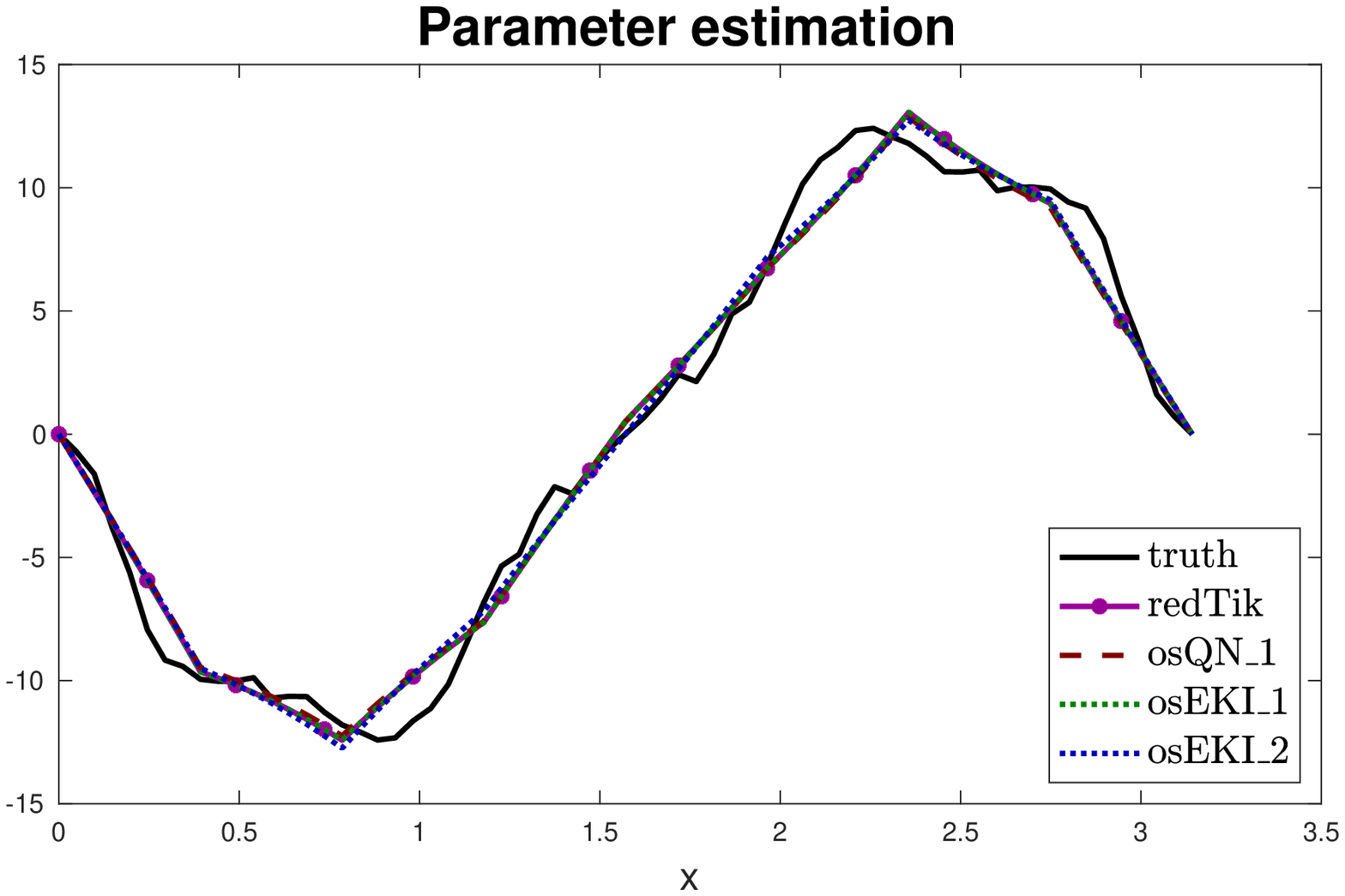}
	\end{subfigure}
	\begin{subfigure}[c]{0.49\textwidth}
	\includegraphics[width=1\textwidth]{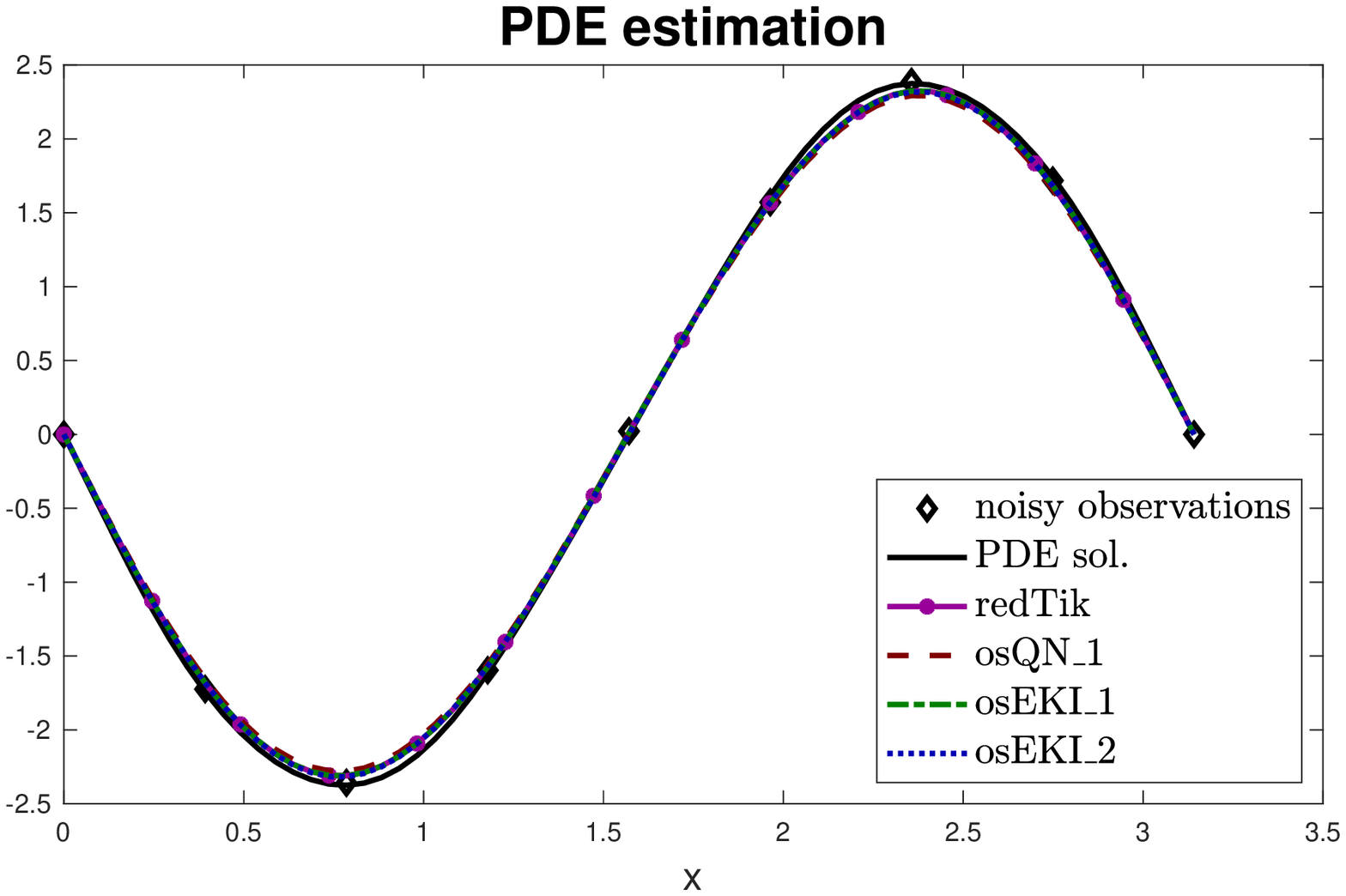}
	\end{subfigure}
    \caption{Comparison of parameter estimation given by EKI with Algorithm
1 (osEKI\_1), the EKI with Algorithm 2 (osEKI\_2) and the quasi-Newton
method with Algorithm 1 (osQN\_1) compared to the Tikhonov solution and the truth (on the left hand side) and in the observation space (on the right hand side).}\label{fig:ex1_est}
\end{figure} 

\begin{figure}[H]
	\begin{subfigure}[c]{0.49\textwidth}
	\includegraphics[width=1\textwidth]{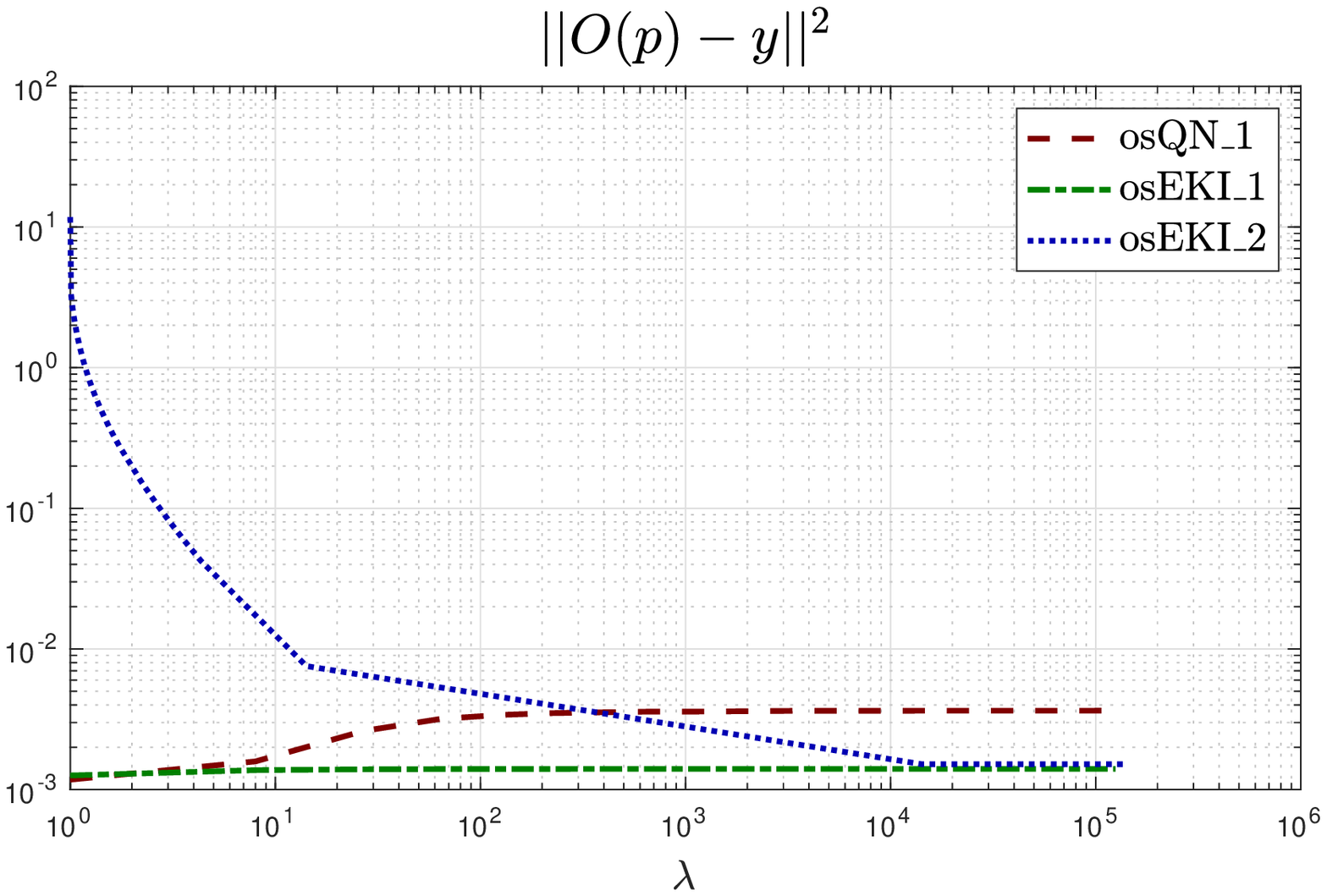}
	\end{subfigure}
	\begin{subfigure}[c]{0.49\textwidth}
	\includegraphics[width=1\textwidth]{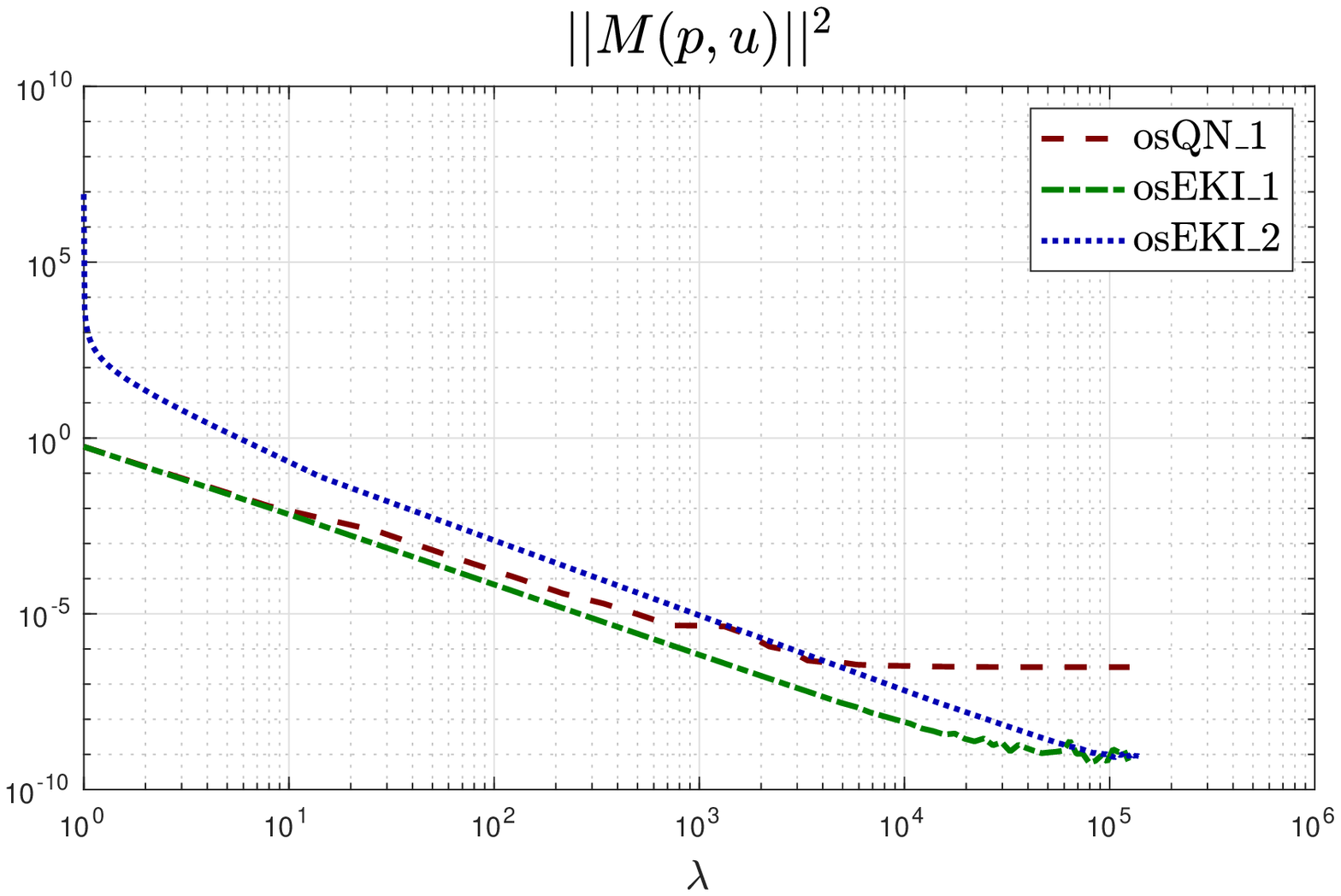}
	\end{subfigure}
    \caption{Comparison of the data misfit given by EKI with Algorithm
1 (osEKI\_1), the EKI with Algorithm 2 (osEKI\_2) and the quasi-Newton
method with Algorithm 1 (osQN\_1) (on the left hand side) and residual of the forward problem (on the right hand side), both w.r.~to $\lambda$.}\label{fig:ex1_datamisfit_PDEerror}
\end{figure} 
The comparison of the data misfit and the residual of the forward problem shown in Figure \ref{fig:ex1_datamisfit_PDEerror} reveals a very good performance of the EKI (for both algorithms) with feasibility of the estimate (w.r.~to the forward problem) in the range of $10^{-10}$.

\subsubsection{One-shot method with neural network approximation}
The next experiment replaces the forward problem by a neural network in the one-shot setting. Due to the excellent performance of Algorithm 2 in the previous experiment, we focus in the following on this approach for the neural network based one-shot inversion.

The EKI for the neural network based one-shot inversion leads to a very good approximation of the regularized solution (cp.~Figure \ref{fig:ex1_est2}), whereas the performance of the quasi-Newton approach is slightly worse, which might be attributed to the nonlinearity introduced by the neural network approximation. 

\begin{figure}[!htb]
	\begin{subfigure}[c]{0.49\textwidth}
	\includegraphics[width=1\textwidth]{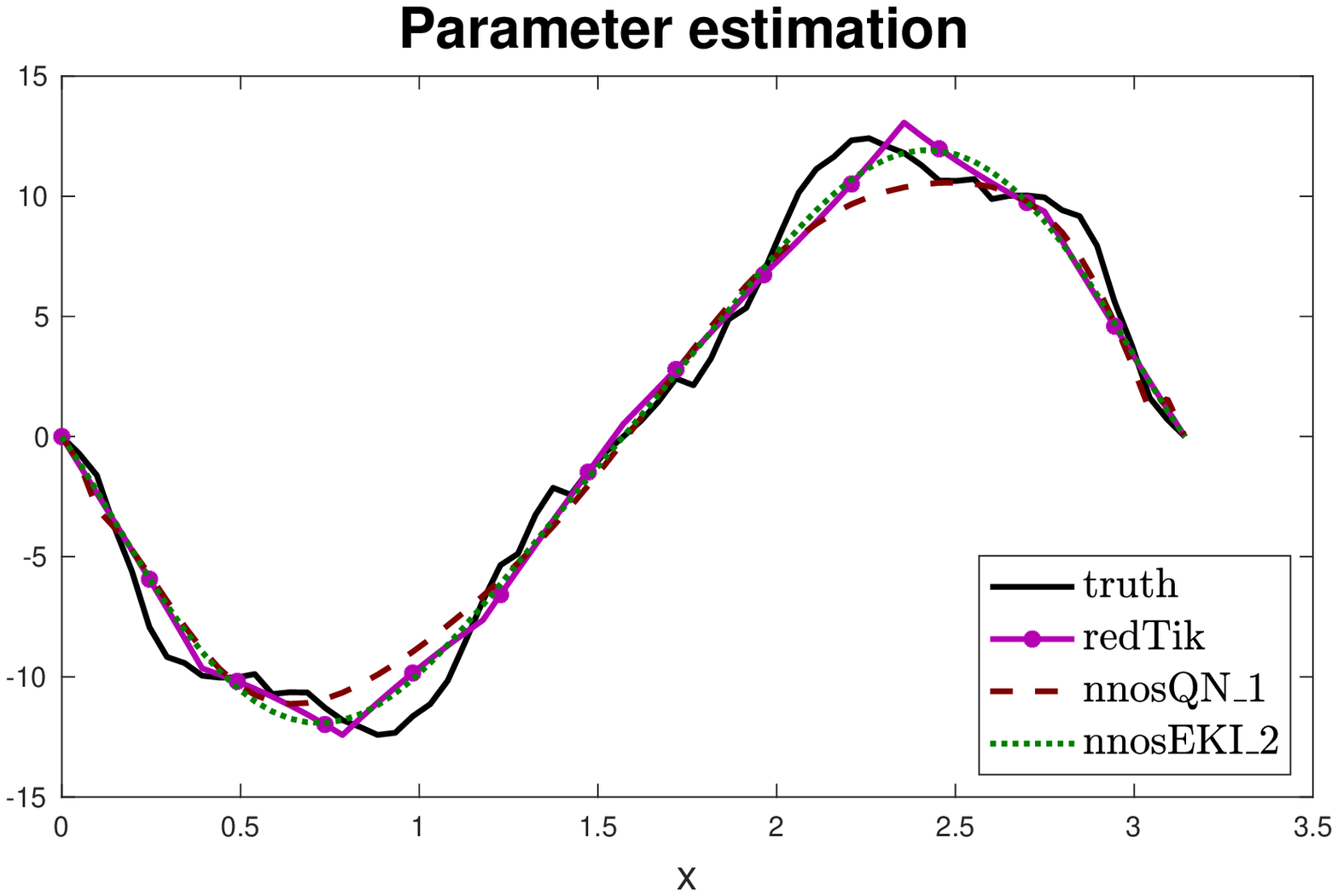}
	\end{subfigure}
	\begin{subfigure}[c]{0.49\textwidth}
	\includegraphics[width=1\textwidth]{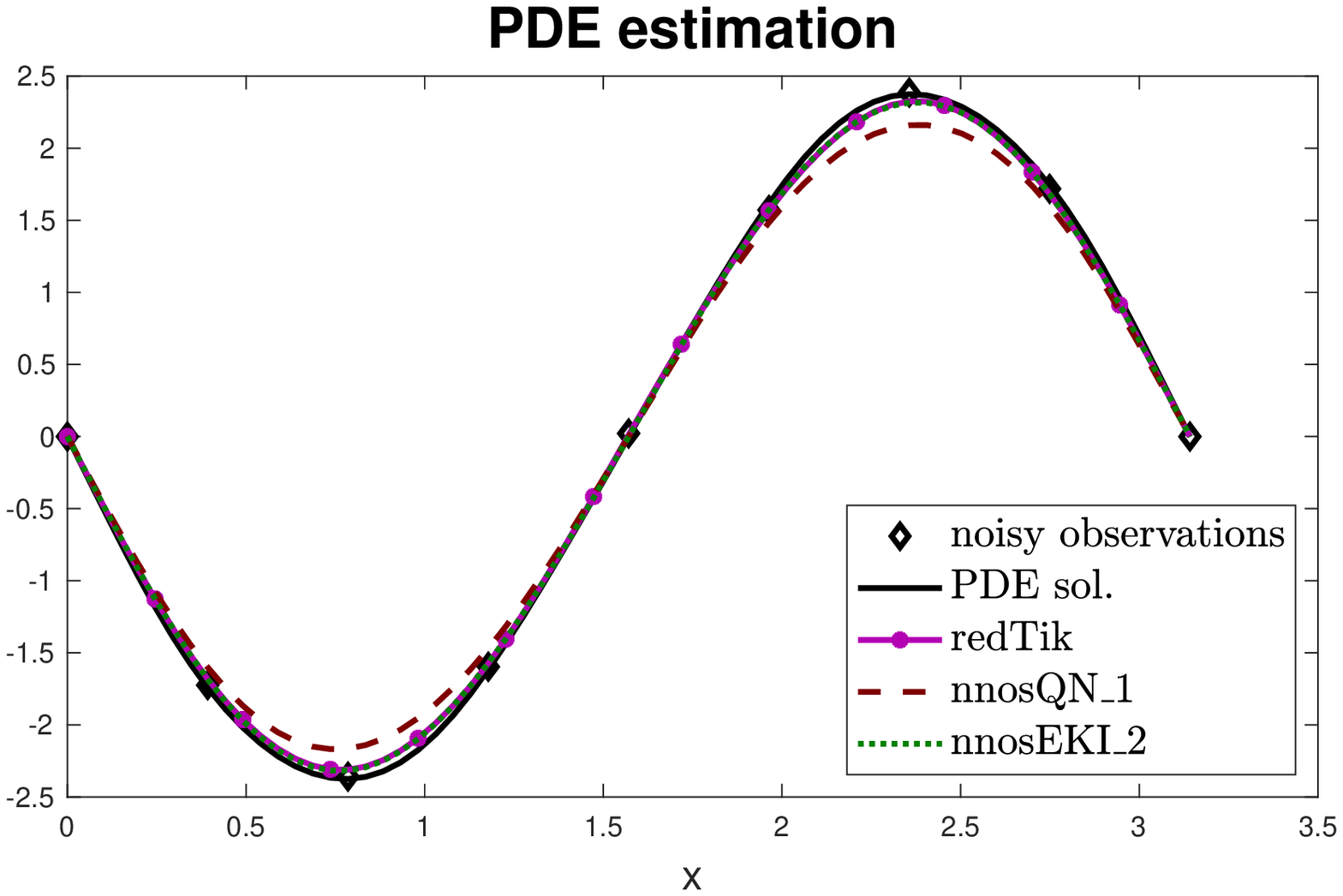}
	\end{subfigure}
    \caption{Comparison of parameter estimation given by the EKI with Algorithm 2 (nnosEKI\_2) and the quasi-Newton
method with Algorithm 1 (nnosQN\_1) for the neural network based one-shot inversion compared to the Tikhonov solution and the truth (on the left hand side) and in the observation space (on the right hand side).}\label{fig:ex1_est2}
\end{figure}

\begin{figure}[!htb]
	\begin{subfigure}[c]{0.49\textwidth}
	\includegraphics[width=1\textwidth]{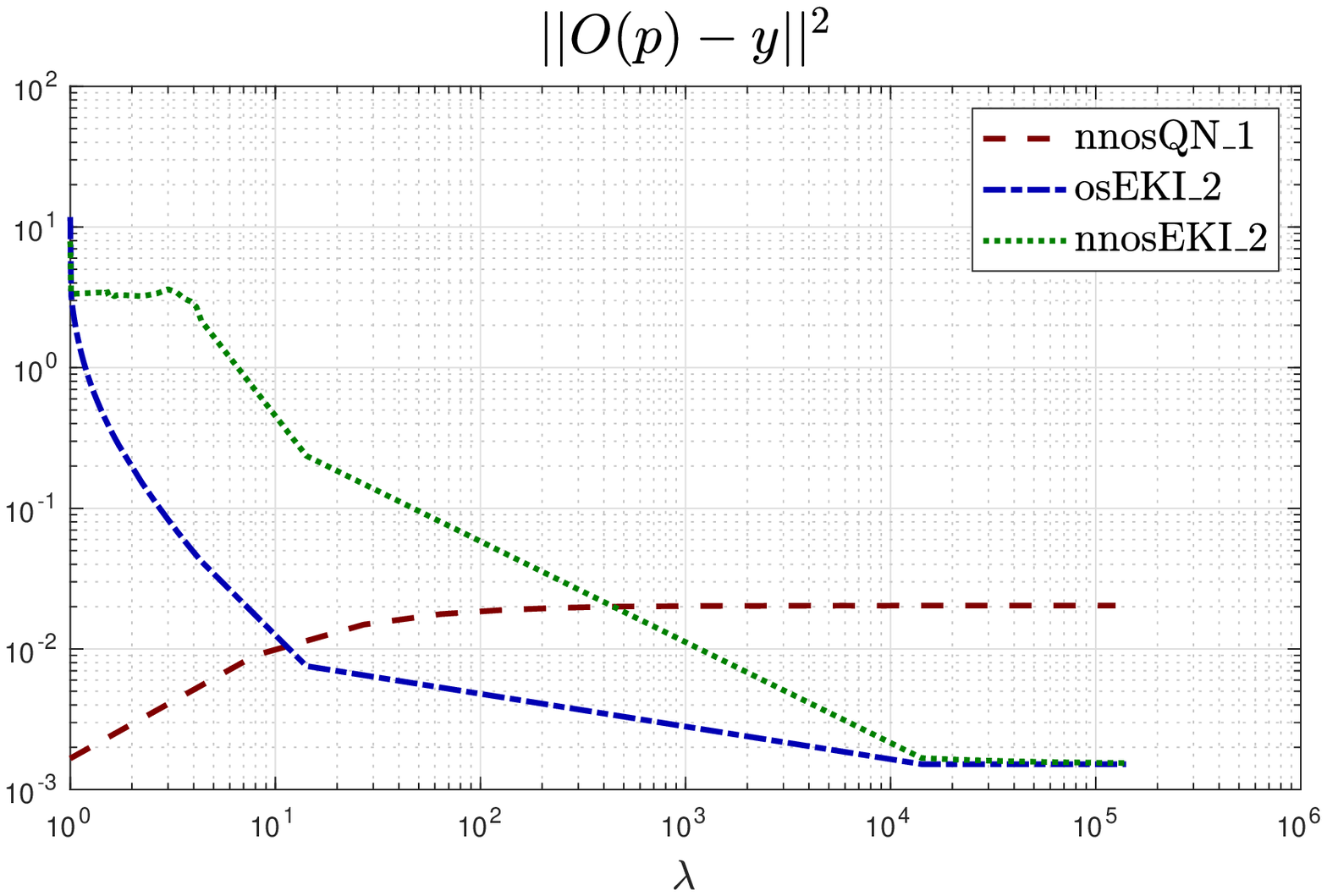}
	\end{subfigure}
	\begin{subfigure}[c]{0.49\textwidth}
	\includegraphics[width=1\textwidth]{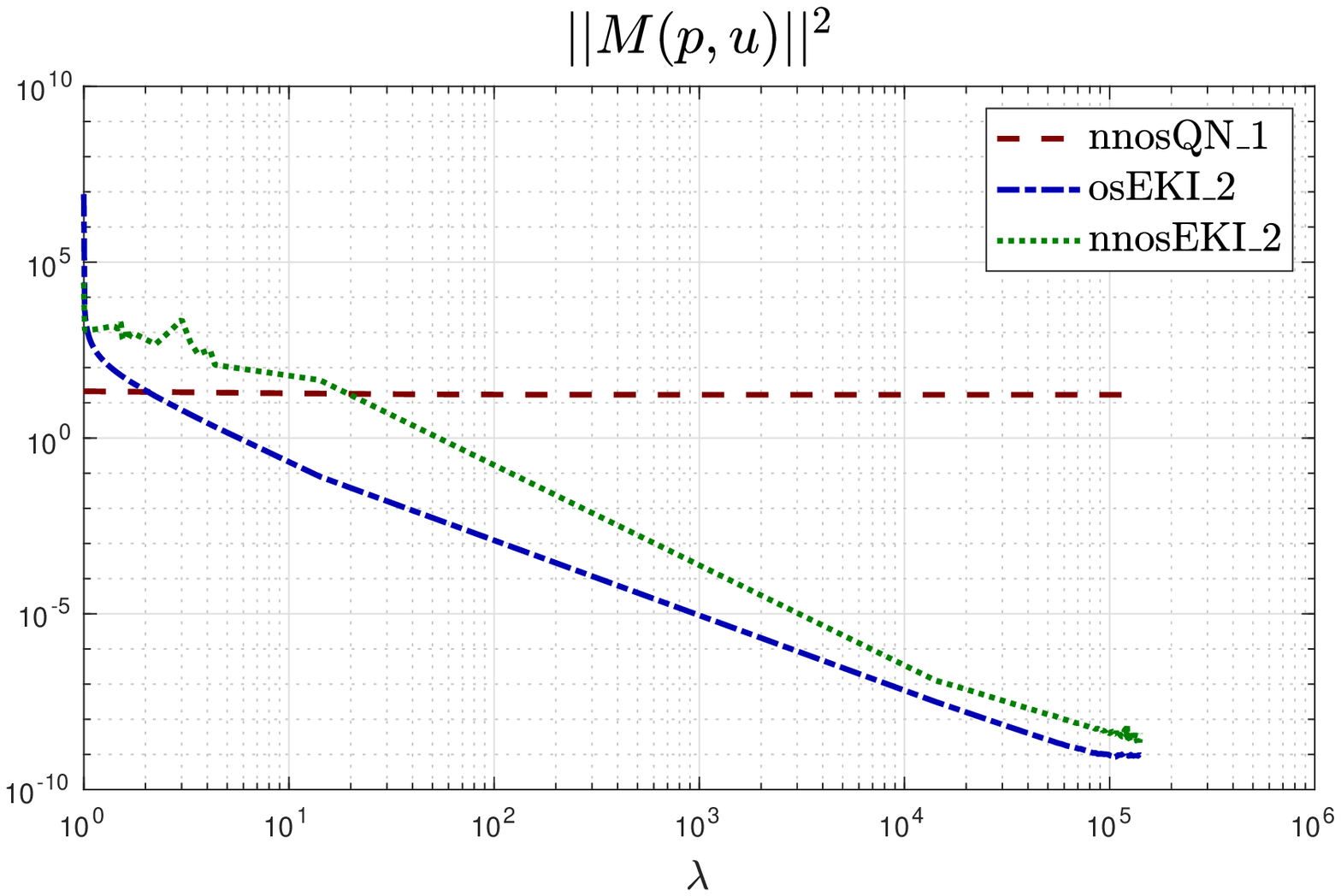}
	\end{subfigure}
    \caption{Comparison of the data misfit given by the EKI with Algorithm 2 (nnosEKI\_2) and the quasi-Newton
method with Algorithm 1 (nnosQN\_1) for the neural network based one-shot inversion compared to EKI with Algorithm
2 (osEKI\_2) from the previous experiment (on the left hand side) and residual of the forward problem (on the right hand side), both w.r.~to $\lambda$.}\label{fig:ex1_est2b}
\end{figure} 

The comparison of the data misfit and residual of the forward problem reveals an excellent convergence behaviour of the EKI for the neural network based one-shot optimization, whereas the quasi-Newton method does not converge to a feasible estimate, cp.~Figure \ref{fig:ex1_est2b}.

\subsubsection{\textcolor{black}{Nonlinear forward model}}
\textcolor{black}{
We consider in the following a nonlinear forward model of the form 
\begin{equation}\label{eq:nonlinear_FP}
\begin{split}
-\nabla\cdot(\exp(u^\dagger)\cdot\nabla p)&=10 \quad \mbox{in} \ D:=(0,\pi), \cr
	p&=0 \quad \mbox{on}\ \partial D\,.
\end{split}
\end{equation}
Note that the mapping from the unknown parameter function to the state is nonlinear.  We use the same discretization as in the linear problem. The unknown parameter $u^\dagger$ is assumed to be Gaussian with zero mean and $C_0=\beta(-\frac{\mathrm{d}^2}{\mathrm{d} x^2})^{-\nu}$ where we choose $\beta=1, \ \nu=2$. Further, we set $\Gamma_{obs}=0.0001\cdot I_{n_y}$, $\hat\Gamma_{model}= 10\cdot I_{n_u}$, $\alpha_1=2$ and $\alpha_2=0$. Furthermore, the structure of the feed-forward DNN remains the same as in the linear case.}

\textcolor{black}{
We compare the one-shot method with neural network approximation resulting from the EKI with Algorithm 2 with the Tikhonov solution of the reduced formulation, which has been approximated by a quasi-Newton method. We determine the scaling parameter $\lambda$ in Algorithm 2 by the ODE $\mathrm d\lambda/\mathrm dt=1$, i.e.~the scaling parameter grows linearly. Similarly to the linear case, we find that the one-shot method with neural network approximation leads to a good approximation of the Tikhonov solution for the reduced model, cp. Figure \ref{fig:ex1_nonlinear_estimates}.
}

\begin{figure}[!htb]
	\begin{subfigure}[c]{0.49\textwidth}
	\includegraphics[width=1\textwidth]{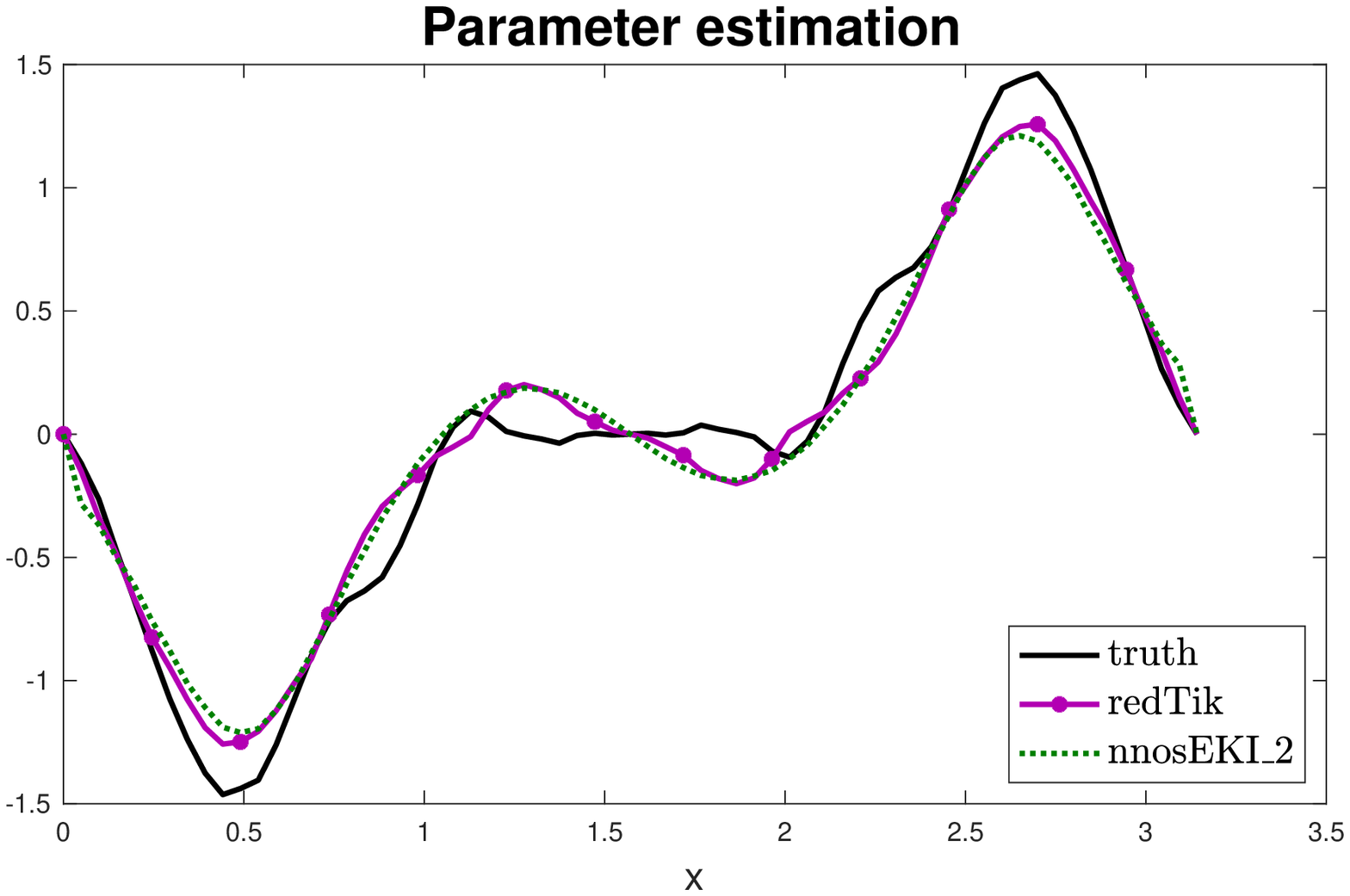}
	\end{subfigure}
	\begin{subfigure}[c]{0.49\textwidth}
	\includegraphics[width=1\textwidth]{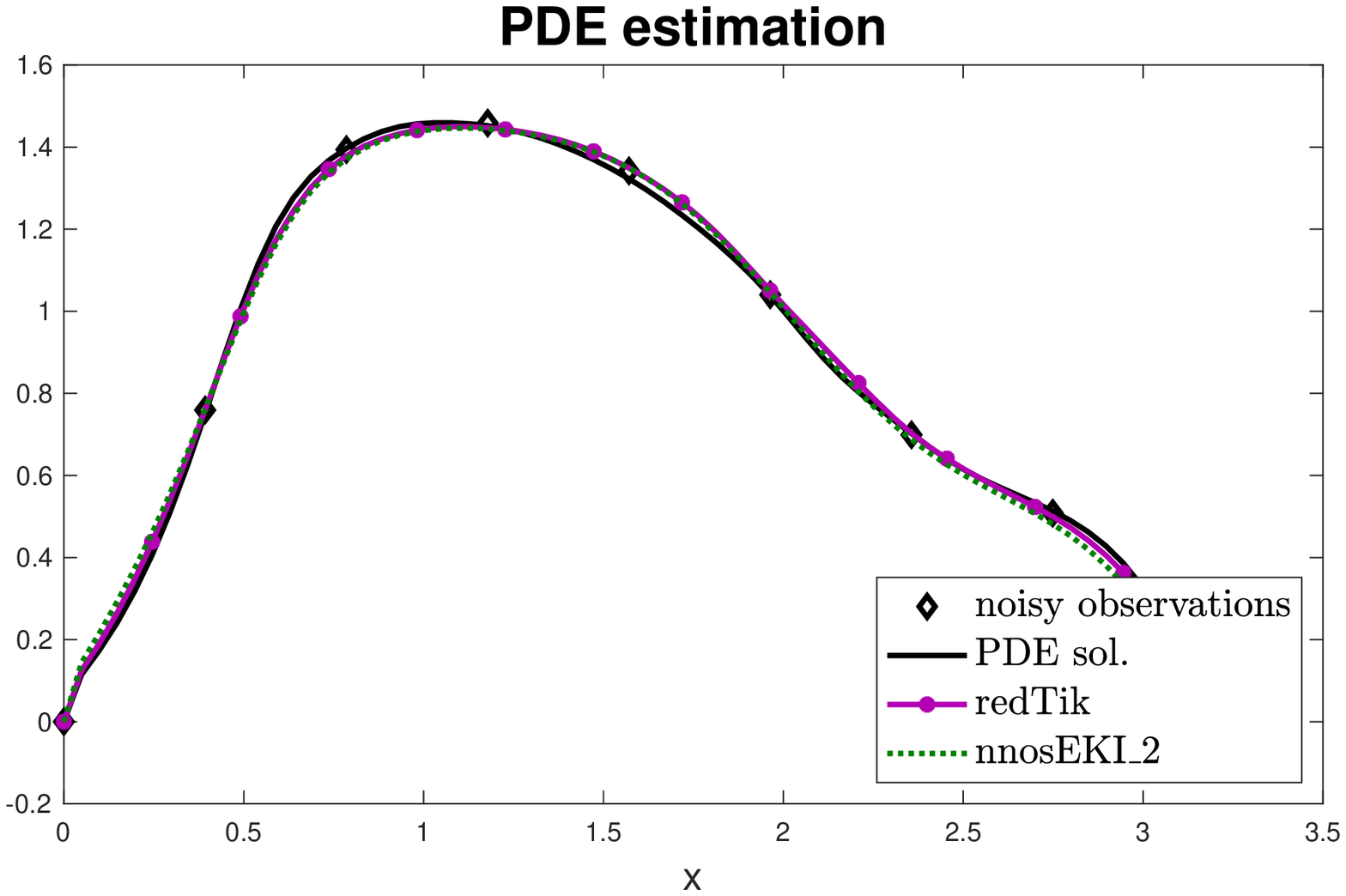}
	\end{subfigure}
    \caption{\textcolor{black}{Comparison of parameter estimation given by the EKI with Algorithm 2 (osEKI\_2) and the Tikhonov solution (on the left hand side) and corresponding PDE solution (on the right hand side).}}\label{fig:ex1_nonlinear_estimates}
\end{figure} 

\textcolor{black}{
In Figure \ref{fig:ex1_nonlinear_datamisfit_PDEerror}, we observe that the penalty parameter $\lambda$ drives the estimate towards feasibilty, i.e. towards the solution of the constrained optimization problem. 
}

\begin{figure}[!htb]
	\begin{subfigure}[c]{0.49\textwidth}
	\includegraphics[width=1\textwidth]{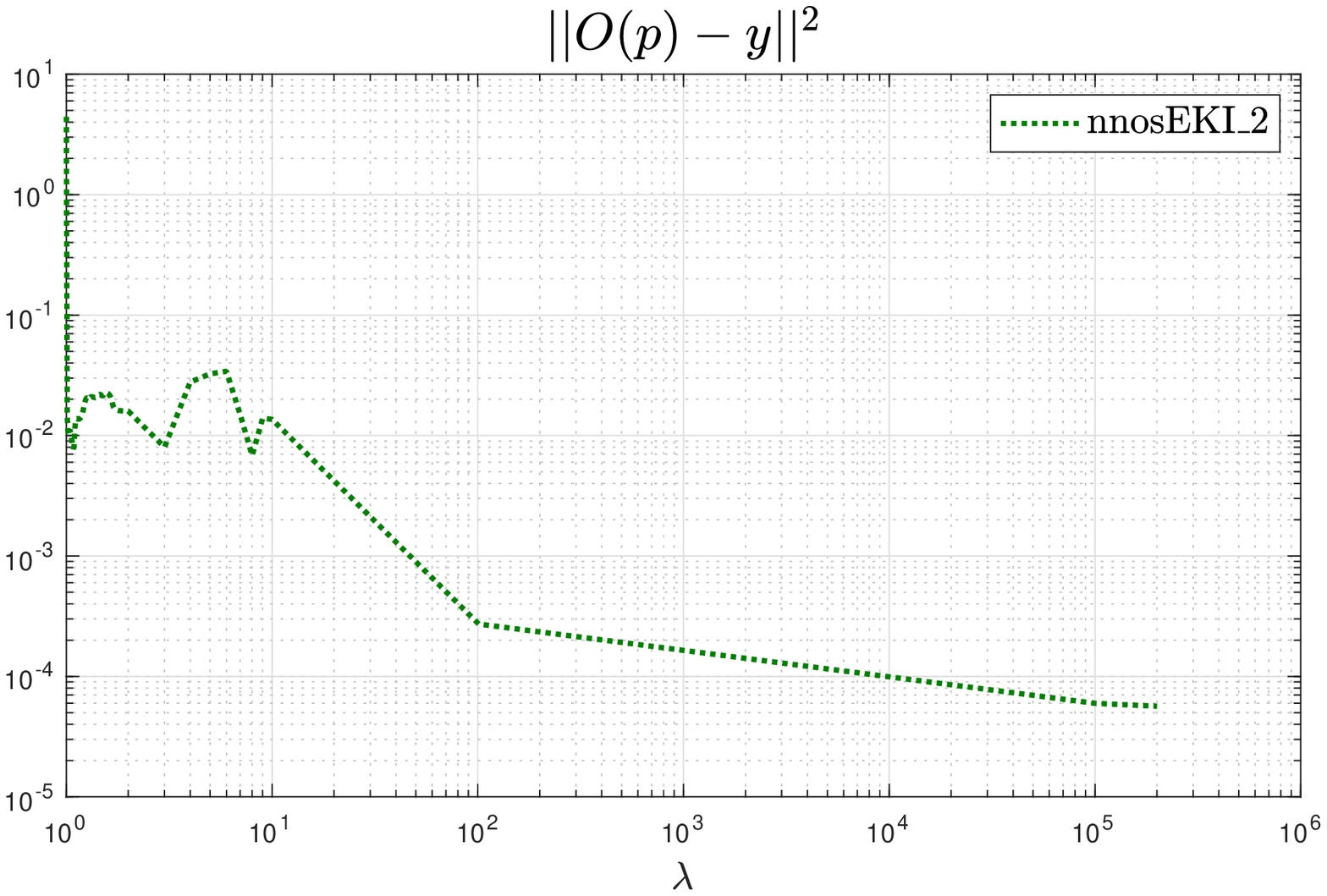}
	\end{subfigure}
	\begin{subfigure}[c]{0.49\textwidth}
	\includegraphics[width=1\textwidth]{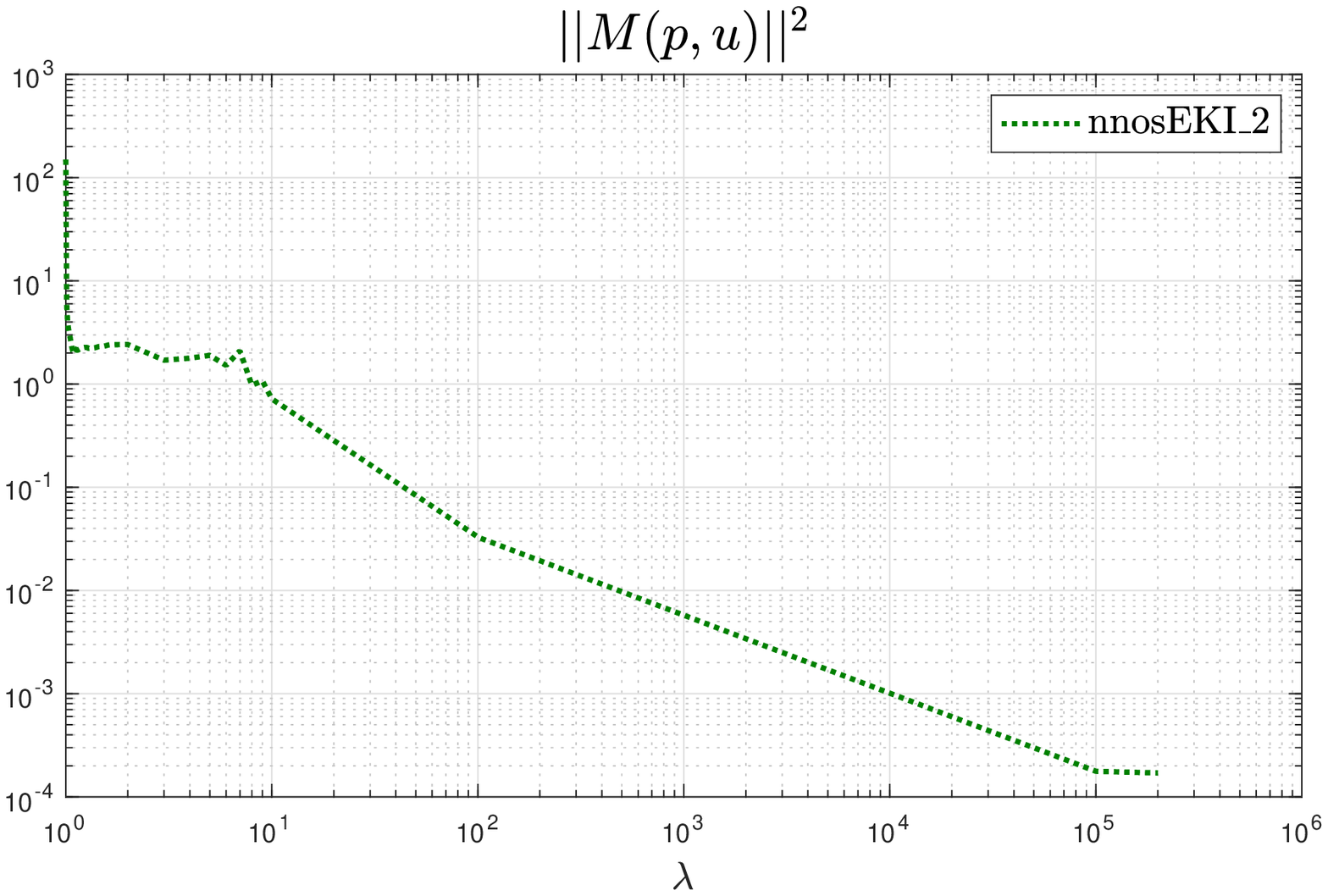}
	\end{subfigure}
    \caption{\textcolor{black}{Data misfit given by the EKI with Algorithm 2 (osEKI\_2) for the neural network based one-shot inversion compared (on the left hand side) and residual of the forward problem (on the right hand side), both w.r.~to $\lambda$.}}\label{fig:ex1_nonlinear_datamisfit_PDEerror}
\end{figure} 

\subsection{Two-dimensional example}

Our second numerical example is based on the two-dimensional Poisson equation
\begin{equation}\label{eq:FP_2d}
\begin{split}
-\Delta p&=u^\dagger \quad \mbox{in} \ D:=(0,1)^2, \cr
	p&=0 \quad \mbox{on}\ \partial D,
\end{split}
\end{equation}
for which we consider again the problem of recovering the unknown source term $u^\dagger$ from noisy observations
\begin{equation}
y = \mathcal O(p^\dagger)+\eta^\dagger,
\end{equation}
with $p^\dagger$ denoting the solution of \eqref{eq:FP_2d}. We consider an observation operator $\mathcal O$ observing $n_y=50$ randomly picked observation points $x_i,\ i=1,\dots,n_y$, shown in Figure~\ref{fig:ex2_reference}. 

We numerically approximate the forward model \eqref{eq:FP_2d} with continuous, piecewise linear finite element basis functions on a mesh with $95$ grid points in $D$ and $40$ grid points on $\partial D$ using the \texttt{MATLAB Partial Differential Equation Toolbox}. We again denote the approximated solution operator by $S\in\R^{n_u\times n_u}$, with $n_u=95$. Similar as before, we assume the unknown parameter $u$ to be Gaussian, this time with (discretized) covariance operator $C_0 = \beta(\tau\cdot \mathrm{id}-\Delta)^{-\nu}$ for $\beta=100$, $\nu=2$ and $\tau=1$. The observational noise covariance  is assumed to be $\Gamma_{obs} = 0.01\cdot I_{n_y}$, whereas we assume a model covariance $\hat\Gamma_{model} = 0.1\cdot I_{n_u}$. We set a regularization parameter $\alpha_1 = 0.002$ and again $\alpha_2=0$. The feed-forward DNN consists of $L=3$ layers, with $N_1=N_2=10$ hidden neurons, $N_0=2$ input neurons and $N_L=1$ output neuron, and sigmoid activation function. The setting of the EKI is as described above with $J=300$ particles drawn as i.i.d. sample from the prior.

Figure \ref{fig:ex2_reference} shows the truth and the corresponding PDE solution.

\begin{figure}[!htb]
	\centering
	\begin{subfigure}[c]{0.49\textwidth}
	\includegraphics[width=1\textwidth]{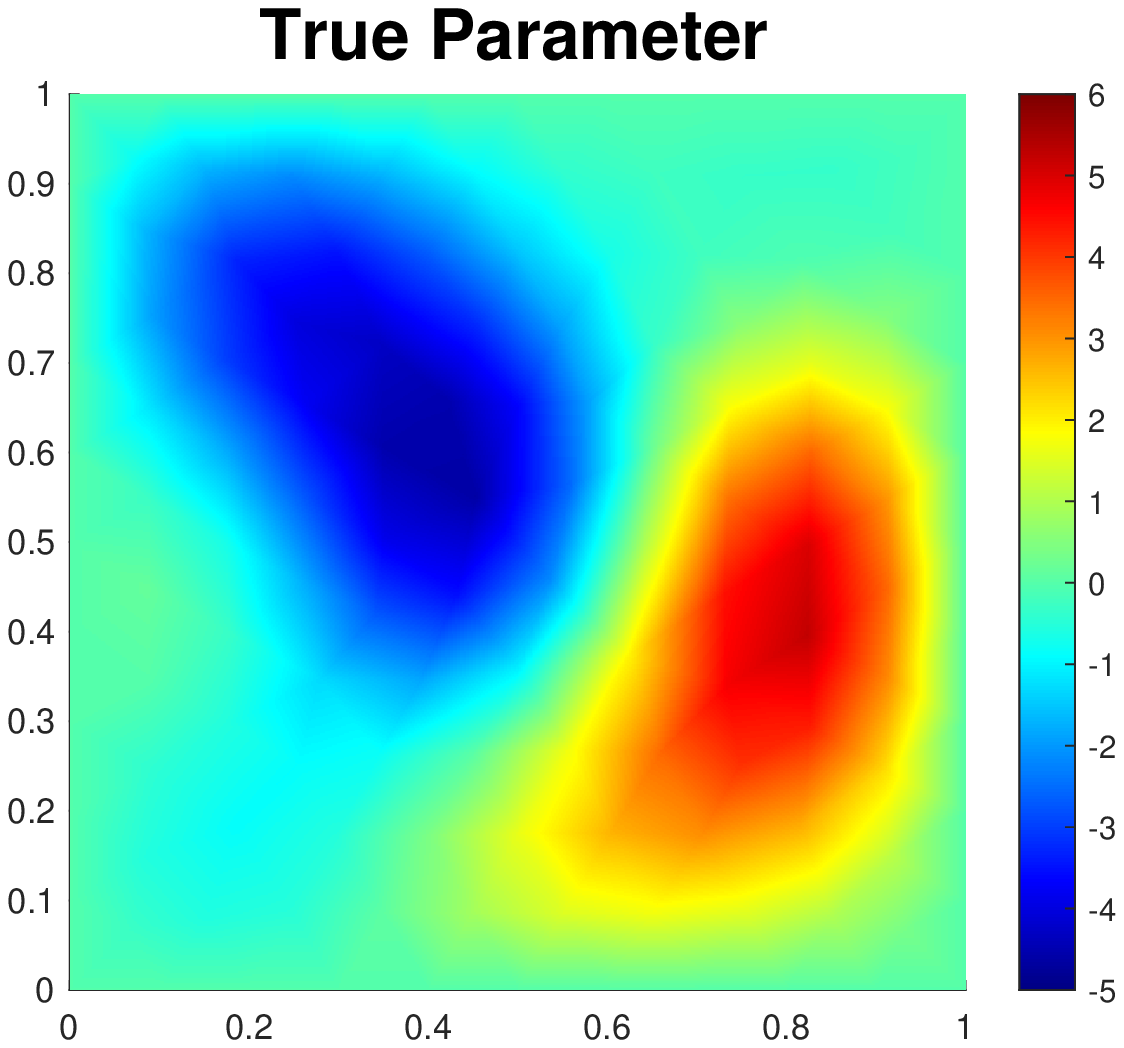}
	\end{subfigure}
	\begin{subfigure}[c]{0.49\textwidth}
	\includegraphics[width=1\textwidth]{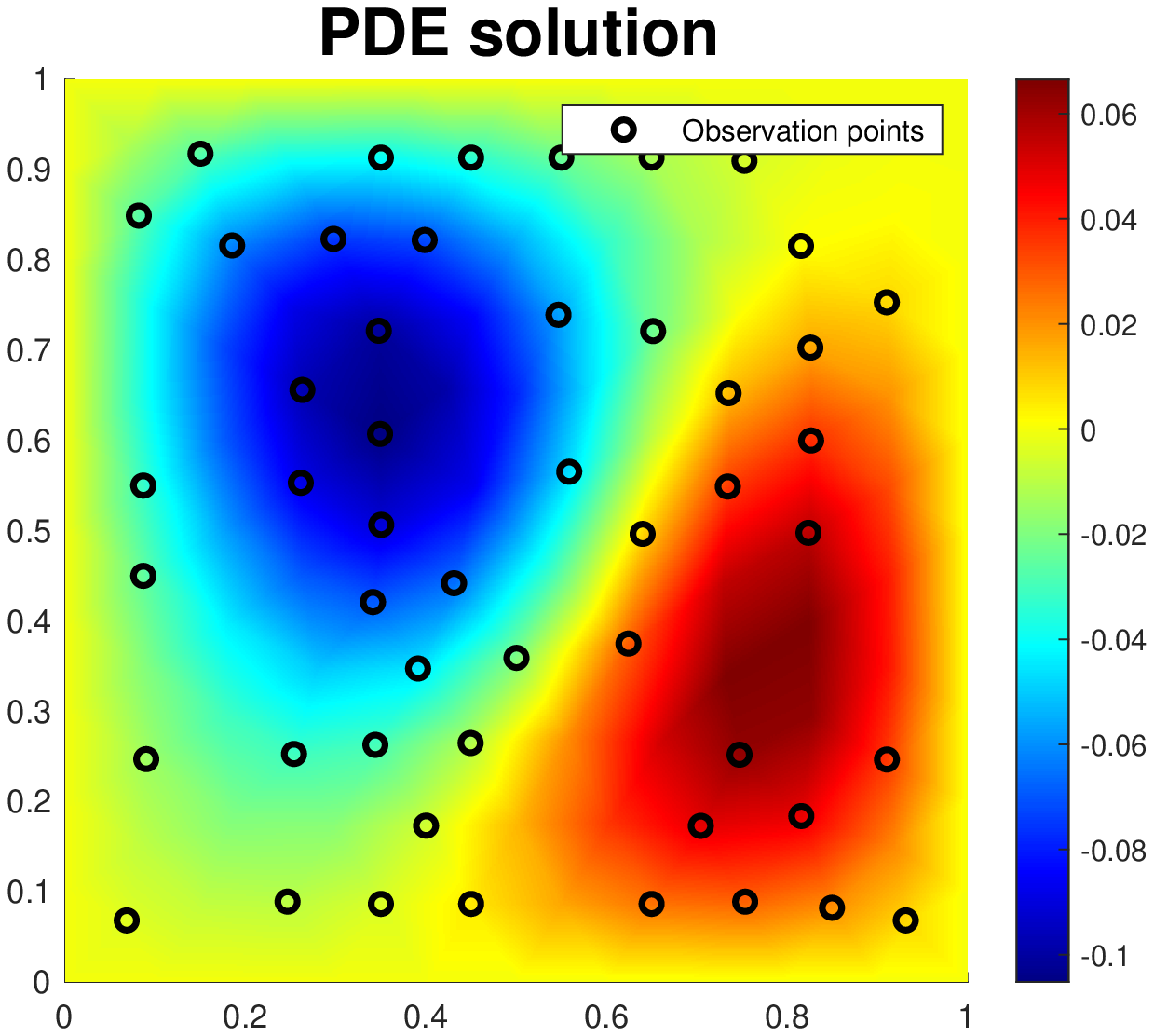}
	\end{subfigure}
    \caption{Ground truth (left hand side) and the corresponding PDE solution (right hand side).}\label{fig:ex2_reference}
\end{figure}

In the following, we compare the neural network based one-shot formulation, solved by the EKI with Algorithm 2, to the explicit Tikhonov solution of the reduced formulation. The scaling parameter $\lambda$ in Algorithm 2 is determined by the ODE $\mathrm d\lambda/\mathrm dt=1/\lambda^2$. Figure \ref{fig:ex2_est} demonstartes that the EKI leads to a comparable solution.


\begin{figure}[!htb]
	\centering
	\begin{subfigure}[c]{0.49\textwidth}
	\includegraphics[width=1\textwidth]{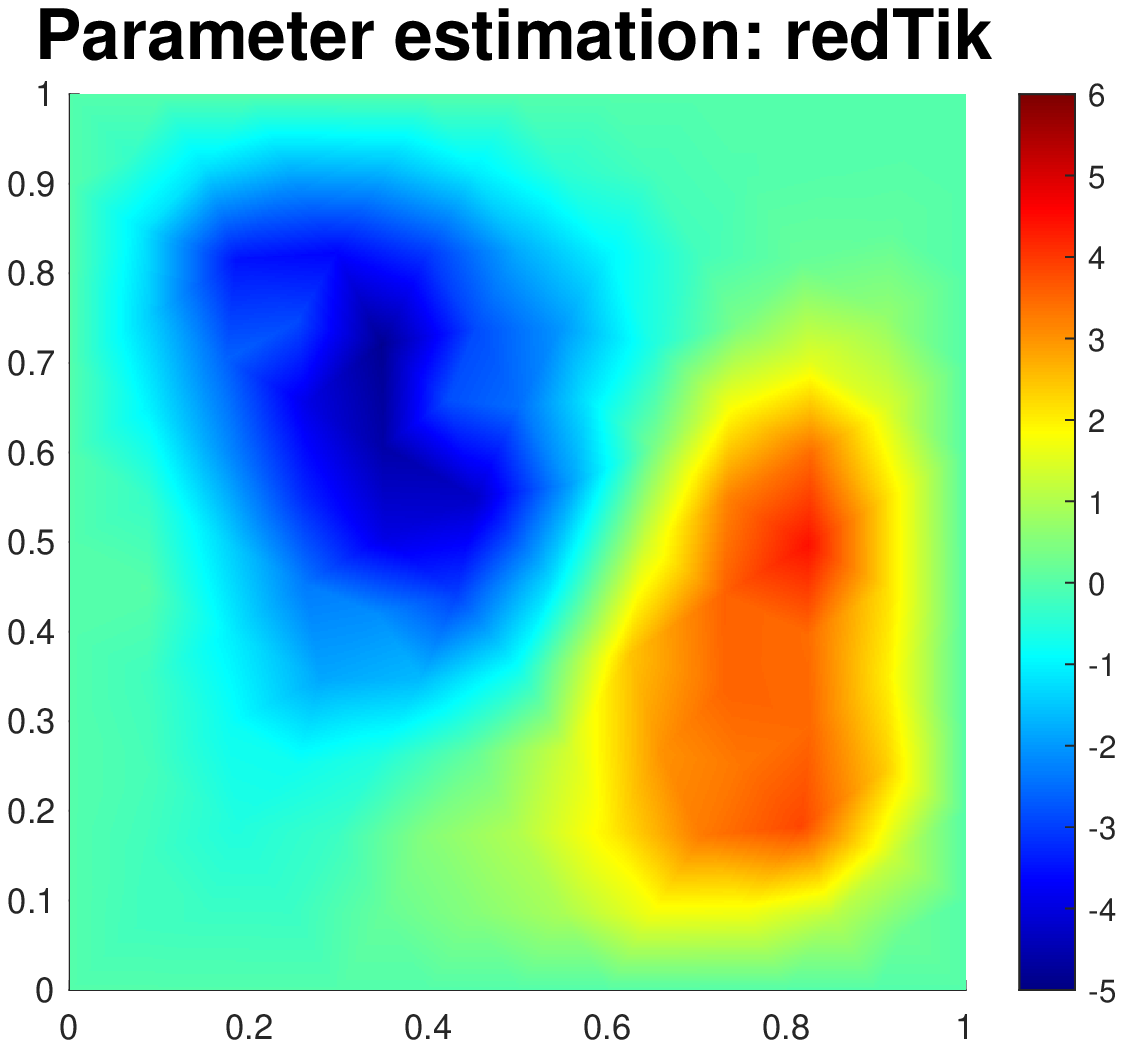}
	\end{subfigure}
	\begin{subfigure}[c]{0.49\textwidth}
	\includegraphics[width=1\textwidth]{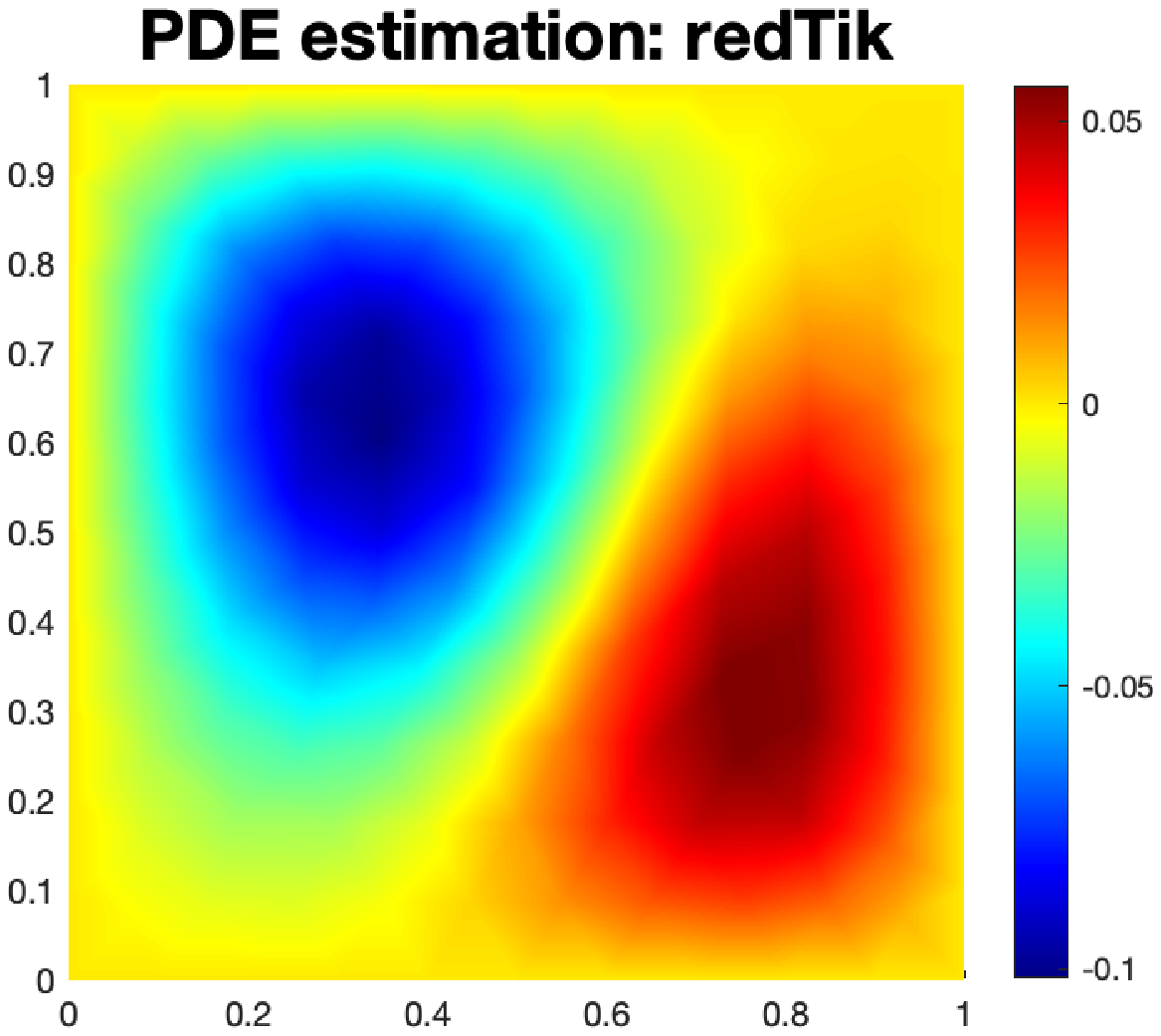}
	\end{subfigure}
	\begin{subfigure}[c]{0.49\textwidth}
	\includegraphics[width=1\textwidth]{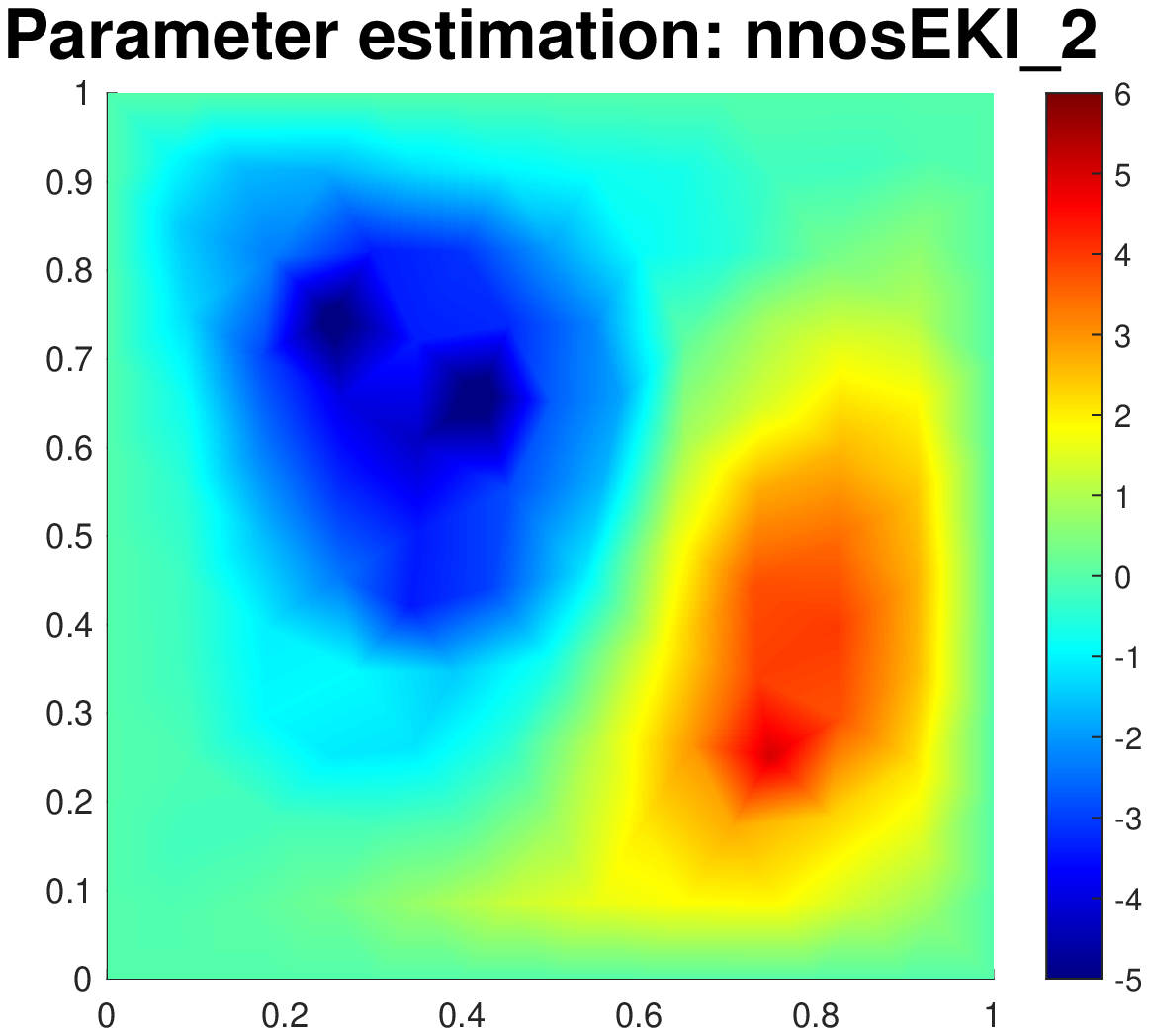}
	\end{subfigure}
	\begin{subfigure}[c]{0.49\textwidth}
	\includegraphics[width=1\textwidth]{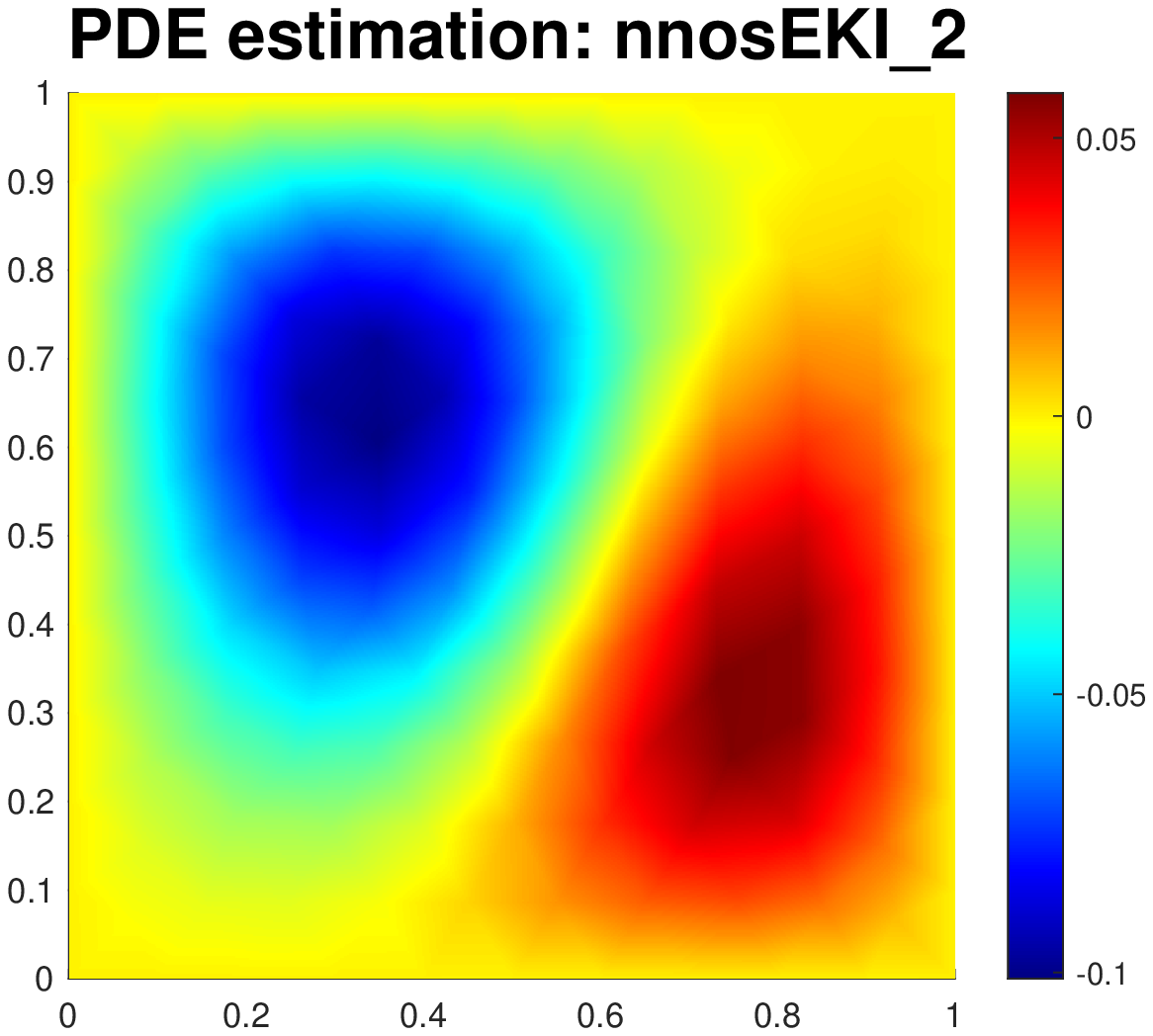}
	\end{subfigure}
    \caption{Comparison of parameter estimation given by the EKI with Algorithm 2 (osEKI\_2) (below) and the Tikhonov solution (above) (on the left hand side) and corresponding PDE solution (on the right hand side).}\label{fig:ex2_est}
\end{figure} 

The proposed approach leads to a feasible solution w.r.~to the forward problem, cp. Figure \ref{fig:ex2_datamisfit_PDEerror2}.
\begin{figure}[!htb]
	\begin{subfigure}[c]{0.49\textwidth}
	\includegraphics[width=1\textwidth]{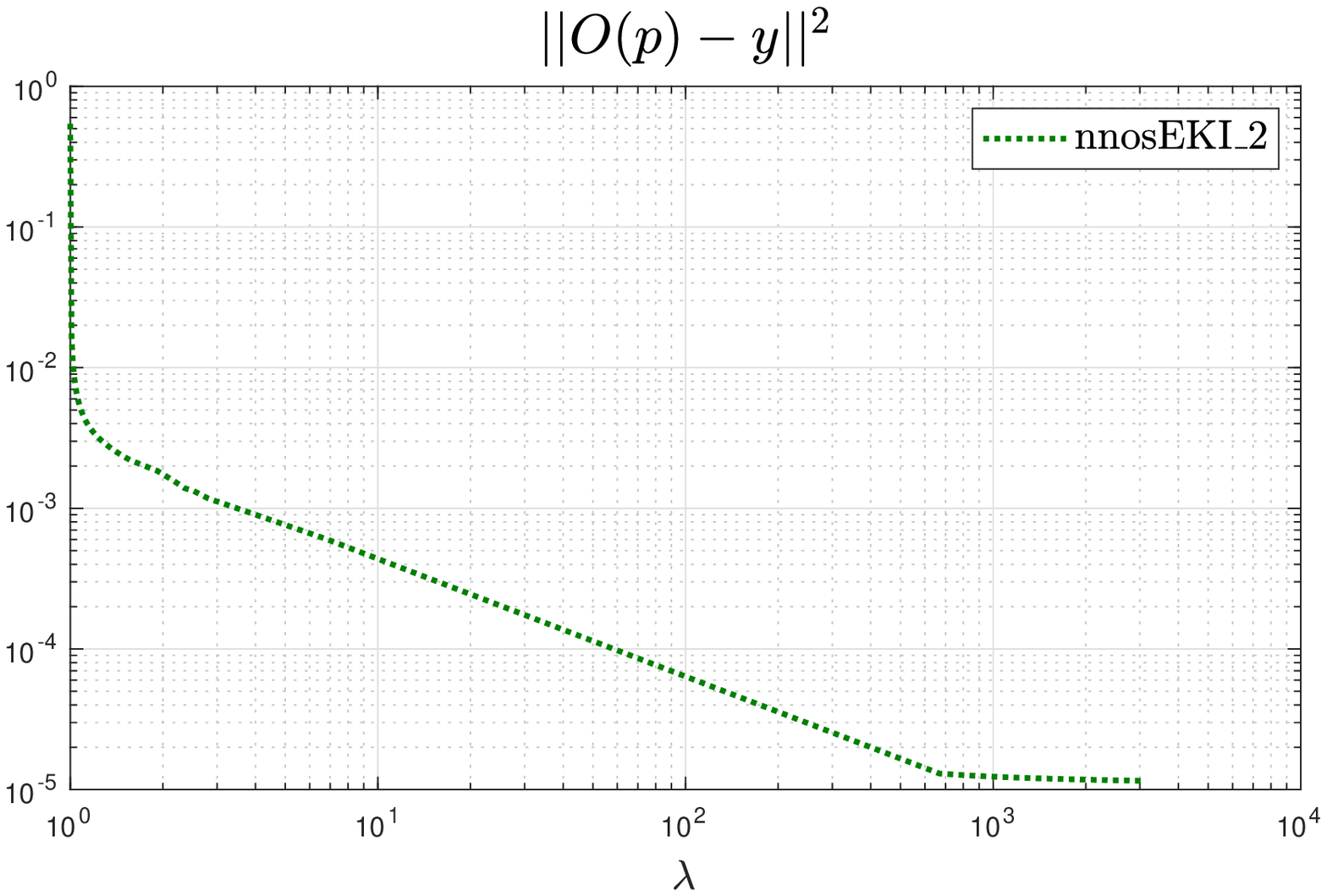}
	\end{subfigure}
	\begin{subfigure}[c]{0.49\textwidth}
	\includegraphics[width=1\textwidth]{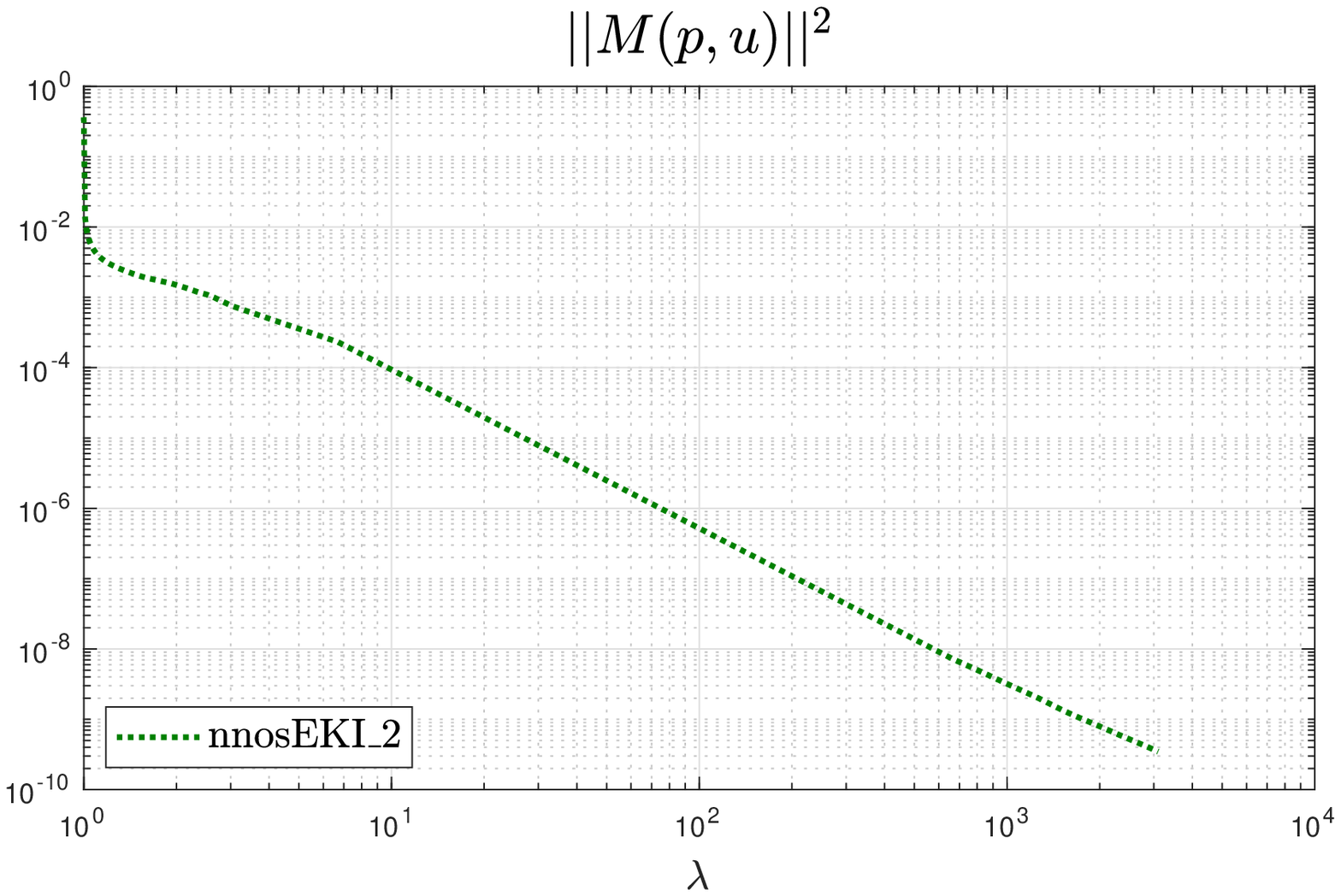}
	\end{subfigure}
    \caption{Data misfit given by the EKI with Algorithm 2 (osEKI\_2) for the neural network based one-shot inversion compared (on the left hand side) and residual of the forward problem (on the right hand side), both w.r.~to $\lambda$.}\label{fig:ex2_datamisfit_PDEerror2}
\end{figure}

\section{Conclusions}\label{sec:concl}
We have demonstrated that the ensemble Kalman inversion for neural network based one-shot inversion is a promising method, regarding  both  estimation quality of the unknown parameter  and  computational  feasibility. The link to the Bayesian setting with vanishing noise allowed to establish a convergence result in a simplified linear setting.   Several directions for future work naturally arise from the presented ideas: theoretical analysis of the neural network based one-shot inversion using recent approximation results of neural networks in the PDE setting, \textcolor{black}{in particular the comparison of the efficiency of neural network and state-of-the-art approximations of the forward model}, parallelization strategies for the EKI using localization ideas, comparison to state-of-the-art optimization methods in the machine learning community.
\subsection*{Acknowledgements}
PG and SW are grateful to the DFG RTG1953 Statistical Modeling of Complex Systems and Processes for funding of their research.

\nocite{ItoKunisch}
\bibliographystyle{plain}
\bibliography{article}

\begin{thebibliography}{10}

\bibitem{Agapiou_2018}
Sergios Agapiou, Martin Burger, Masoumeh Dashti, and Tapio Helin.
\newblock Sparsity-promoting and edge-preserving maximum a posteriori
  estimators in non-parametric {B}ayesian inverse problems.
\newblock {\em Inverse Problems}, 34(4):045002, feb 2018.

\bibitem{actaip}
Simon Arridge, Peter Maass, Ozan Öktem, and Carola-Bibiane Schönlieb.
\newblock Solving inverse problems using data-driven models.
\newblock {\em Acta Numerica}, 28:1–174, 2019.

\bibitem{bertsekas}
Dimitri~P. Bertsekas.
\newblock {\em Nonlinear Programming}.
\newblock Athena Scientific, 1999.

\bibitem{BSWW19}
Dirk Blömker, Claudia Schillings, Philipp Wacker, and Simon Weissmann.
\newblock Well posedness and convergence analysis of the ensemble {K}alman
  inversion.
\newblock {\em Inverse Problems}, 35(8):085007, jul 2019.

\bibitem{volker}
Alfio Borzi and Volker Schulz.
\newblock {\em Computational optimization of systems governed by partial
  differential equations}.
\newblock Society for Industrial and Applied Mathematics, USA, 2012.

\bibitem{CST19}
Neil~K. Chada, Andrew~M. Stuart, and Xin~T. Tong.
\newblock Tikhonov regularization within ensemble {K}alman inversion.
\newblock {\em SIAM Journal on Numerical Analysis}, 58, 01 2019.

\bibitem{chada2019convergence}
Neil~K. Chada and Xin~T. Tong.
\newblock Convergence acceleration of ensemble {K}alman inversion in nonlinear
  settings, 2019.

\bibitem{Clason}
Christian Clason, Tapio Helin, Remo Kretschmann, and Petteri Piiroinen.
\newblock Generalized modes in {B}ayesian inverse problems.
\newblock {\em SIAM/ASA Journal on Uncertainty Quantification}, 7(2):652--684,
  2019.

\bibitem{Dashti_2013}
Masoumeh Dashti, Kody J.~H. Law, Andrew~M. Stuart, and Jochen Voss.
\newblock {MAP} estimators and their consistency in {B}ayesian nonparametric
  inverse problems.
\newblock {\em Inverse Problems}, 29(9):095017, sep 2013.

\bibitem{DS17}
Masoumeh Dashti and Andrew~M. Stuart.
\newblock {\em The {B}ayesian approach to inverse problems}, pages 311--428.
\newblock Springer International Publishing, Cham, 2017.

\bibitem{BEG19}
Dennis~Maximilian Elbr\"{a}chter, Julius Berner, and Philipp Grohs.
\newblock {\em How degenerate is the parametrization of neural networks with
  the {R}e{LU} activation function?}, pages 7790--7801.
\newblock Curran Associates, Inc., 2019.

\bibitem{EHN96}
Heinz~Werner Engl, Martin Hanke, and Andreas Neubauer.
\newblock {\em Regularization of inverse problems}.
\newblock Mathematics and Its Applications. Springer Netherlands, 1996.

\bibitem{doi:10.1137/140981319}
Oliver~G. Ernst, Bj\"orn Sprungk, and Hans-J\"org Starkloff.
\newblock Analysis of the ensemble and polynomial chaos {K}alman filters in
  {B}ayesian inverse problems.
\newblock {\em SIAM/ASA Journal on Uncertainty Quantification}, 3(1):823--851,
  2015.

\bibitem{GE03}
Geir Evensen.
\newblock The ensemble {K}alman filter: theoretical formulation and practical
  implementation.
\newblock {\em Ocean Dynamics}, 53(4):343--367, Nov 2003.

\bibitem{garbunoinigo2019interacting}
Alfredo Garbuno-Inigo, Franca Hoffmann, Wuchen Li, and Andrew~M. Stuart.
\newblock Interacting {L}angevin diffusions: {G}radient structure and ensemble
  {K}alman sampler.
\newblock {\em SIAM Journal on Applied Dynamical Systems}, 19(1):412--441,
  2020.

\bibitem{Goodfellow-et-al-2016}
Ian Goodfellow, Yoshua Bengio, and Aaron Courville.
\newblock {\em Deep Learning}.
\newblock MIT Press, 2016.
\newblock \url{http://www.deeplearningbook.org}.

\bibitem{HLR18}
Eldad Haber, Felix Lucka, and Lars Ruthotto.
\newblock Never look back - {A} modified {E}n{KF} method and its application to
  the training of neural networks without back propagation.
\newblock {\em arXiv preprint arXiv:1805.08034}, 2018.

\bibitem{Helin_2015}
Tapio Helin and Martin Burger.
\newblock Maximum a posteriori probability estimates in infinite-dimensional
  {B}ayesian inverse problems.
\newblock {\em Inverse Problems}, 31(8):085009, jul 2015.

\bibitem{HSZ20}
Lukas Herrmann, Christoph Schwab, and Jakob Zech.
\newblock Deep {R}e{LU} neural network expression rates for data-to-{Q}o{I}
  maps in {B}ayesian {PDE} inversion.
\newblock Technical Report 2020-02, Seminar for Applied Mathematics, ETH
  Z{\"u}rich, Switzerland, 2020.

\bibitem{Higdon}
Dave Higdon, Marc Kennedy, James~C. Cavendish, John~A. Cafeo, and Robert~D.
  Ryne.
\newblock Combining field data and computer simulations for calibration and
  prediction.
\newblock {\em SIAM J. Sci. Comput.}, 26(2):448?466, February 2005.

\bibitem{ILS13}
Marco~A. Iglesias, Kody J.~H. Law, and Andrew~M. Stuart.
\newblock Ensemble {K}alman methods for inverse problems.
\newblock {\em Inverse Problems}, 29(4):045001, 2013.

\bibitem{ItoKunisch}
Kazufumi Ito and Karl Kunisch.
\newblock {\em Lagrange multiplier approach to variational problems and
  applications}, volume~15 of {\em Advances in Design and Control}.
\newblock Society for Industrial and Applied Mathematics (SIAM), Philadelphia,
  PA, 2008.

\bibitem{KS10}
Jari Kaipio and Erkki Somersalo.
\newblock {\em Statistical and computational inverse problems}.
\newblock Applied mathematical sciences; Volume 160. Springer Science \&
  Business Media, New York, NY, 2010.

\bibitem{K16}
Barbara Kaltenbacher.
\newblock Regularization based on all-at-once formulations for inverse
  problems.
\newblock {\em SIAM J. Numerical Analysis}, 54:2594--2618, 2016.

\bibitem{K17}
Barbara Kaltenbacher.
\newblock All-at-once versus reduced iterative methods for time dependent
  inverse problems.
\newblock {\em Inverse Problems}, 33(6):064002, may 2017.

\bibitem{Kaltenbacher2008IterativeRM}
Barbara Kaltenbacher, Andreas Neubauer, and Otmar Scherzer.
\newblock Iterative regularization methods for nonlinear ill-posed problems.
\newblock In {\em Radon Series on Computational and Applied Mathematics}, 2008.

\bibitem{OHagan}
Marc~C. Kennedy and Anthony O'Hagan.
\newblock Bayesian calibration of computer models.
\newblock {\em Journal of the Royal Statistical Society: Series B (Statistical
  Methodology)}, 63(3):425--464, 2001.

\bibitem{KS19}
Nikola~B. Kovachki and Andrew~M. Stuart.
\newblock Ensemble {K}alman inversion: a derivative-free technique for machine
  learning tasks.
\newblock {\em Inverse Problems}, 35(9):095005, aug 2019.

\bibitem{KPRS19}
Gitta Kutyniok, Philipp Petersen, Mones Raslan, and Reinhold Schneider.
\newblock A theoretical analysis of deep neural networks and parametric {PDE}s.
\newblock {\em arXiv preprint arXiv:1904.00377}, 2019.

\bibitem{OPS19}
Joost A.~A. Opschoor, Philipp~C. Petersen, and Christoph Schwab.
\newblock Deep {R}e{LU} networks and high-order finite element methods.
\newblock {\em Analysis and Applications}, 12 2019.

\bibitem{RPK17}
Maziar Raissi, Paris Perdikaris, and George~Em Karniadakis.
\newblock Physics informed deep learning ({P}art {I}): {D}ata-driven solutions
  of nonlinear partial differential equations.
\newblock {\em arXiv preprint arXiv:1711.10561}, 2017.

\bibitem{RPK17_2}
Maziar Raissi, Paris Perdikaris, and George~Em Karniadakis.
\newblock Physics informed deep learning ({P}art {II}): {D}ata-driven discovery
  of nonlinear partial differential equations.
\newblock {\em arXiv preprint arXiv:1711.10566}, 2017.

\bibitem{2019JCoPh.378..686R}
Maziar {Raissi}, Paris {Perdikaris}, and George~Em {Karniadakis}.
\newblock {Physics-informed neural networks: A deep learning framework for
  solving forward and inverse problems involving nonlinear partial differential
  equations}.
\newblock {\em Journal of Computational Physics}, 378:686--707, February 2019.

\bibitem{SS17}
Claudia Schillings and Andrew~M. Stuart.
\newblock Analysis of the ensemble {K}alman filter for inverse problems.
\newblock {\em SIAM Journal on Numerical Analysis}, 55(3):1264--1290, 2017.

\bibitem{SS17b}
Claudia Schillings and Andrew~M. Stuart.
\newblock Convergence analysis of ensemble {{K}alman} inversion: the linear,
  noisy case.
\newblock {\em Applicable Analysis}, 97(1):107--123, 2018.

\bibitem{S12}
Christoph Schwab.
\newblock {QMC} {G}alerkin discretization of parametric operator equations.
\newblock In Josef Dick, Frances~Y. Kuo, Gareth~W. Peters, and Ian~H. Sloan,
  editors, {\em Monte Carlo and Quasi-Monte Carlo Methods 2012}, pages
  613--629, Berlin, Heidelberg, 2013. Springer Berlin Heidelberg.

\bibitem{SZ19}
Christoph Schwab and Jakob Zech.
\newblock Deep learning in high dimension: {N}eural network expression rates
  for generalized polynomial chaos expansions in {UQ}.
\newblock {\em Analysis and Applications}, 17(01):19--55, 2019.

\bibitem{SDK2020}
Yeonjong {Shin}, Jerome {Darbon}, and George~Em {Karniadakis}.
\newblock {On the convergence and generalization of physics informed neural
  networks}.
\newblock {\em arXiv preprint arXiv:2004.01806}, 2020.

\bibitem{AMS10}
Andrew~M. Stuart.
\newblock Inverse problems: {A} {B}ayesian perspective.
\newblock {\em Acta Numerica}, 19:451–559, 2010.

\bibitem{YMK20}
Liu Yang, Xuhui Meng, and George~Em Karniadakis.
\newblock {B}-{PINN}s: {B}ayesian physics-informed neural networks for forward
  and inverse {PDE} problems with noisy data.
\newblock {\em arXiv preprint arXiv:2003.06097v1}, 2020.

\bibitem{Y17}
Dmitry Yarotsky.
\newblock Error bounds for approximations with deep {R}e{LU} networks.
\newblock {\em Neural Networks}, 94:103--114, 2017.

\end{thebibliography}

\end{document}